\documentclass[12pt,leqno]{amsart}
\usepackage{amsmath,amssymb,amsfonts}
\usepackage{eucal}
\usepackage[colorlinks=false,breaklinks=true,pdftex]{hyperref}

\newtheorem{lem}{Lemma}[section]
\newtheorem{cor}[lem]{Corollary}
\newtheorem{prop}[lem]{Proposition}
\newtheorem{thm}[lem]{Theorem}

\theoremstyle{definition}
\newtheorem{exam}{Example}[section]
\newcommand \myexam[1]{\begin{exam}[#1]}

\numberwithin{equation}{section}

\renewcommand{\phi}{\varphi}                 
\renewcommand{\epsilon}{\varepsilon}
\newcommand\eset{\varnothing}
\newcommand\setm{\setminus}

\newcommand\bigtimes{\mathop{ \raisebox{-.15\height}{\huge$\times$} }}

\newcommand\inv{^{-1}}
\newcommand\textb{\text{\rm b}}
\newcommand\pib{\pi_{\textb}}

\newcommand\Lat{\operatorname{Lat}}
\newcommand\Latb{\operatorname{\Lat^{\textb}}}
\newcommand\clos{\operatorname{cl}}
\newcommand\rk{\operatorname{rk}}
\newcommand \ind{{:}}

\newcommand\ba{\mathbf{a}}

\newcommand\bc{\mathbf{c}}
\newcommand\bg{\mathbf{g}}
\newcommand\bh{\mathbf{h}}
\newcommand\bm{\mathbf{m}}
\newcommand\bu{\mathbf{u}}

\newcommand\bw{\mathbf{w}}
\newcommand\bx{\mathbf{x}}
\newcommand\by{\mathbf{y}}

\newcommand\bbR{\mathbb{R}}
\newcommand\bbZ{\mathbb{Z}}

\newcommand\1{{1_\fG}}
\newcommand \fC{\mathfrak C}
\newcommand \fG{\mathfrak G}
\newcommand \fM{\mathfrak M}

\newcommand \fW{\mathfrak W}

\newcommand \cA{\mathcal A}
\newcommand \cB{\mathcal B}

\newcommand \cP{\mathcal P}
\newcommand \cPf{\cP_\text{fin}}
\newcommand \cPb{\mathcal P_{\textb}}
\newcommand \cS{\mathcal S}

\renewcommand\H{H}
\newcommand\Phih{{(\Phi,h)}}
\newcommand\G{\Gamma}
\newcommand\IA{\operatorname{IA}}
\newcommand\EA{\operatorname{EA}}
\newcommand\II{\operatorname{I\mspace{1mu}I}}
\newcommand\EI{\operatorname{EI}}

\begin{document}

\begin{center}

\large
\textsc{
Lattice Points in Orthotopes \\and a Huge Polynomial Tutte Invariant \\ of Weighted Gain Graphs}\\[20pt]
Version of \today
\normalsize
\vskip20pt

{David Forge\footnote{Part of the research of this author was performed while visiting the State University of New York at Binghamton in 2005.  Part of his research was supported by the TEOMATRO project, grant number ANR-10-BLAN 0207, in 2013.}\\
Laboratoire de recherche en informatique UMR 8623\\
B\^at.\ 490, Universit\'e Paris-Sud\\
91405 Orsay Cedex, France\\
E-mail: {\tt forge@lri.fr}}\\[10pt]

and\\[10pt]

{Thomas Zaslavsky\footnote{Some of the work of this author was performed while visiting the Laboratoire de recherche en informatique, Universit\'e Paris-Sud, Orsay, in 2007.}\\
Department of Mathematical Sciences \\
State University of New York at Binghamton\\
Binghamton, NY 13902-6000, U.S.A.\\
E-mail: {\tt zaslav@math.binghamton.edu}}\\[20pt]

\end{center}

\small
{\sc Abstract.}
A \emph{gain graph} is a graph whose edges are orientably labelled from a group.  A \emph{weighted gain graph} is a gain graph with vertex weights from an abelian semigroup, where the gain group is lattice ordered and acts on the weight semigroup.  
For weighted gain graphs we establish basic properties and we present general dichromatic and forest-expansion polynomials that are Tutte invariants (they satisfy Tutte's deletion-contraction and multiplicative identities).  
In order to do that we develop a relative of the Tutte polynomial of a semimatroid.  
Our dichromatic polynomial includes the classical graph one by Tutte, Zaslavsky's two for gain graphs, Noble and Welsh's for graphs with positive integer weights, and that of rooted integral gain graphs by Forge and Zaslavsky.  
It is not a universal Tutte invariant of weighted gain graphs; that remains to be found.  

An evaluation of one example of our polynomial counts proper list colorations of the gain graph from a color set with a gain-group action.
When the gain group is $\bbZ^d$, the lists are order ideals in the integer lattice $\bbZ^d$, and there are specified upper bounds on the colors, then there is a formula for the number of bounded proper colorations that is a piecewise polynomial function of the upper bounds, of degree $nd$ where $n$ is the order of the graph.

This example leads to graph-theoretical formulas for the number of integer lattice points in an orthotope but outside a finite number of affinographic hyperplanes, and for the number of $n \times d$ integral matrices that lie between two specified matrices but not in any of certain subspaces defined by simple row equations.  


 \normalsize\\

\thanks{Version of \today.}

\thanks{The research of the first author was performed, for the most part, while visiting the State University of New York at Binghamton.}

\vskip 20pt

\tableofcontents

\section{Introduction}\label{intro}

An \emph{integral orthotope} is a rectangular parallelepiped with integral vertices and edges parallel to the coordinate axes.  
An \emph{integral affinographic hyperplane} is a hyperplane of the form $x_j = x_i + a$, where $a$ is an integer.  
(All our orthotopes and affinographic hyperplanes will be integral.)  
We wish to count the points of the integer lattice that lie in an orthotope but not in any of a given \emph{arrangement} (a finite set) of affinographic hyperplanes.  
An arrangement is \emph{centered} if there is a point common to all hyperplanes.

\begin{thm}\label{T:geomorthotope}
Let $P := [0,m_1]\times\cdots\times[0,m_n]$, where $\bm=(m_1,\ldots,m_n) \in \bbZ_{\geq0}^n$, be an integral orthotope in $\bbR^n$ and let $\cA$ be an arrangement of affinographic hyperplanes.  
The number of integer points in $P \setm \bigcup\cA$ equals 
\begin{align*}
\sum_{\cB \subseteq \cA: \text{centered}} (-1)^{|\cB|} f_\cB(\bm)  ,
\end{align*}
where each $f_\cB$ is a function that depends linearly on the $m_i$ if they are sufficiently large.
\end{thm}

Theorem \ref{T:geomorthotope} can be regarded as a highly generalized graph coloring theorem.  The best way to see that is to specialize it to ordinary graphs.  
Let $\chi_\Gamma(m_1,\ldots,m_n)$ be the number of proper colorations of $\Gamma$ using the color set $\{0,1,\ldots,m_i\}$ for vertex $v_i$, where the $m_i$ are independently chosen nonnegative integers.  
A \emph{graphic hyperplane} is a hyperplane of the form $x_i=x_j$; call it $h_{ij}$.  (``Affinographic'' hyperplanes are affine transforms of graphic hyperplanes.)  The hyperplane arrangement corresponding to a simple graph $\Gamma$, say of order $n$, is $\cA := \{ h_{ij}: v_iv_j\in E(\Gamma)\}$.  
Thus, a proper coloration is the same as an integral point in $P := [0,m_1]\times\cdots\times[0,m_n]$ but not in any of the hyperplanes of $\cA$.
Let $\mu_\Gamma$ denote the M\"obius function of $\Lat\Gamma$, the lattice of closed subsets of $E(\Gamma)$ in the usual graphic matroid; $\Lat\Gamma$ is isomorphic to the lattice of intersections of hyperplanes in $\cA$.  For an edge set $B$, $\pi(B)$ is the partition of $V(\Gamma)$ induced by the components of $B$.  
The graphic specialization (which as far as we know is new) is this:

\begin{cor}\label{C:geomorthotope}
Let $m_1,\ldots,m_n \in \bbZ_{\geq0}$ and let $\Gamma$ be a simple graph of order $n$.  The number of proper colorations of $\Gamma$ such that the color of vertex $v_i$ is in the set $\{0,1,\ldots,m_i\}$ is
\begin{align*}
\chi_\Gamma(m_1,\ldots,m_n)
 &= \sum_{B \subseteq E(\Gamma)} (-1)^{|B|} \prod_{W_k \in \pi(B)}  ( 1 + \min_{v_i\in W_k} m_i)  \\
 &= \sum_{B \in \Lat\Gamma} \mu_\Gamma(\eset,B) \prod_{W_k \in \pi(B)}  ( 1 + \min_{v_i\in W_k} m_i)  .
\end{align*}
\end{cor}

If all $m_i=m$, the color set is $\{0,1,\ldots,m\}$ and $\chi_\Gamma(m,\ldots,m)$ equals the ordinary chromatic polynomial $\chi_\Gamma(m+1)$; then Corollary \ref{C:geomorthotope} reduces to well-known formulas for the ordinary chromatic polynomial.  Thus, the corollary is a limited list-coloring generalization of standard results.  In Corollary \ref{C:geomspec} we generalize to coloring from arbitrary finite lists.

Theorem \ref{T:geomorthotope} generalizes the main theorem of \cite{SOA}, which covered the case of hypercubes, where all $m_i=m$.  
It is a simplified version of Theorem \ref{T:geomorthotopemu}, which gives the exact form of $f_\cB$.  The proof is carried out in terms of proper colorations of \emph{weighted gain graphs}.  
Briefly, a gain graph is a graph whose edges are orientably labelled by elements of a group; that is, reversing the edge inverts the gain.  In a weighted gain graph the group is lattice ordered and the vertices carry weights from an abelian semigroup on which the gain group has an action.  For instance, the gain group and weight semigroup may be $(\bbZ,+)$ and $(\bbZ,\max)$.  
Theorem \ref{T:geomorthotope} comes from that kind of weighted gain graph: a vertex corresponds to a coordinate in $\bbR^n$ and an edge $v_iv_j$ with gain $a$ in the indicated direction corresponds to a hyperplane $x_j = x_i + a$.  (In the opposite direction, $v_jv_i$, the gain is $-a$.)

Our aim is to understand the mathematics behind the main result of \cite{SOA}, as well as Theorem \ref{T:geomorthotope} and more general results like Theorem \ref{T:geomorthospec}, by interpreting lattice-point--counting functions as Tutte invariants of weighted gain graphs.  A consequence is a great generalization of the techniques and results of \cite{SOA}, which also incorporates as a special case the theorem of Noble and Welsh \cite{NW} on graphs with positive integral weights.  (During the delayed revision of this paper for publication we found out about an independent generalization of \cite{NW} by Ellis-Monaghan and Moffatt \cite{EMM}; it is our generalization without gains, but with edge weights.  Their purpose is an application to statistical physics, for which they do not need gains.  For excluding lattice points in affinographic hyperplanes, gains are essential.)

Our ``huge polynomial Tutte invariant'' is $Q_\Phih(\bu,v,z)$, the total dichromatic polynomial of a weighted gain graph $\Phih$.  (For the definition see Equation \eqref{E:dichromaticpoly}.)  It is, first of all, a common generalization of the ordinary and balanced dichromatic polynomials of a gain graph from \cite[Part III]{BG}, which are direct generalizations of Tutte's dichromatic polynomial $Q_\G(u,v)$ of a graph \cite{TutteODP}.  In a weighted gain graph the variable $u$ of the standard dichromatic polynomial splits into a multitude of independent variables $u_h$, one for each possible vertex weight $h$; $\bu$ denotes the totality of all $u_h$.  Calling this polynomial a Tutte invariant is justified by the fact that $Q_\Phih(\bu,v,z)$ satisfies the standard deletion-contraction, multiplication, and normalization identities of the Tutte polynomial of an ordinary graph or matroid.  (The identities are stated in Section \ref{invariant}.)  

We obtain counting functions by substituting particular values for the weight variables $u_h$.  For instance, in Theorem \ref{T:geomorthotope} the weight semigroup is $\{[l,m]: l \leq m \in \bbZ\} \cup \{\eset\}$ with intersection as its operation and we get the lattice-point counting function by setting $u_{[l,m]} = -(m-l+1)$ and $u_\eset = 0$ (see Section \ref{affino}).
We even have a formula for the number of $n \times d$ integral matrices in an integral orthotope in $(\bbZ^d)^n$ that do not lie in any of a class of $d$-codimensional subspaces whose equations compare rows of the matrix (Section \ref{affinomatrix}).  

\begin{exam}\label{X:Qorder2i}
We do a very small example, just to show what a total dichromatic polynomial looks like.  
Our weighted gain graph $\Phih$ has $n=2$ vertices.  The gain group is $(\bbZ^2,+)$.  The weight semigroup is $(\bbZ^2,\max)$.  
The edge set is
$$
E = \{e_1=(0,0)(v_1,v_2),\ e_2=(2,0)(v_1,v_2),\ e_3=(-1,2)(v_1,v_2)\},
$$  
where the notation $(2,0)(v_1,v_2)$ means an edge with endpoints $v_1$ and $v_2$, whose gain is $(2,0)$ (in the additive group $\bbZ^2$) in the direction from $v_1$ to $v_2$ and consequently $(-2,0)$ in the opposite direction.
The weights of the vertices are $h_1 = (2,0)$ and $h_2 = (-1,3)$.
The total dichromatic polynomial is  
\begin{equation}\label{E:Qorder2i}
Q_\Phih(\bu,v,z) = u_{(2,0)}u_{(-1,3)} + u_{(2,3)} + u_{(4,3)} + u_{(22,3)} + 3z + vz .
\end{equation}
The $u_{(x,y)}$ are weight variables; there is one for each $(x,y)\in\bbZ^2$ but, as here, only finitely many appear in any one polynomial $Q_\Phih(\bu,v,z)$.  
Setting $z=1$ or $0$ gives the dichromatic and balanced dichromatic polynomials mentioned above.  

We derive \eqref{E:Qorder2i} in Example \ref{X:Qorder2}, after we explain the techniques.

Geometrically, the edges in this example, whose gains are 2-dimensional vectors rather than scalars as with the affinographic hyperplanes introduced at the beginning, correspond to subspaces of $(\bbR^2)^n$ (where $n=2$) whose codimension is 2.   Write the coordinates of $(\bbR^2)^2$ in the form $(\bx_1,\bx_2)\in(\bbR^2)^2$; that is, with $\bx_1,\bx_2\in\bbR^2$.   
The edges correspond to subspaces whose equations are, respectively, $\bx_2-\bx_1=(0,0)$, $\bx_2-\bx_1=(2,0)$, and $\bx_2-\bx_1=(-1,2)$.  

This example, with its multidimensional gains, illustrates the general integer-lattice type in Example \ref{X:Zd} and foreshadows the theory of Section \ref{affino}.
\end{exam}

We had hoped that the total dichromatic polynomial $Q_\Phih$ would be a universal Tutte invariant of weighted gain graphs, in the sense that every Tutte invariant is an evaluation of it.  However, we are sure it is not.  It takes no account of balanced loops and loose edges.  When we attempted to improve it, even without weights, by incorporating variables for those edges, there were relations amongst the variables that suggest a universal Tutte invariant more general than $Q_\Phih$ exists but is very complicated.  Finding that universal invariant, even for gain graphs without weights, is certain to be very challenging.

We proceed now to a section of technical definitions and then to the definition and properties of the total dichromatic polynomial, including a spanning-forest expansion (but not a spanning-tree expansion; that would be a semimatroid Tutte polynomial, for which we refer the reader to Ardila \cite{ArdilaSemi}).  
Then in Section \ref{coloration} we specialize to counting proper colorations and the application to lattice-point--counting problems.  We conclude with the connection to Noble and Welsh \cite{NW}.


\section{Weighted gain graphs}\label{defs}

Notation:  For a real number $x$ we write 
$$x^+ := \max(0,x), \qquad x^- := -\min(0,x),$$
the positive and negative parts of $x$ (so $x=x^+-x^-$).  Applied to a vector $\bx=(x_1,\ldots,x_d)$, $\bx^+:=(x_1^+,\ldots,x_d^+)$ and $\bx^-$ is similar (so $\bx=\bx^+-\bx^-$).  

\subsection{Graphs}\label{graphs}

Edges of a graph $\G=(V,E)$ are of four kinds.  A \emph{link} has two distinct endpoints; a \emph{loop} has two coinciding endpoints.  
A \emph{half edge} has one endpoint, and a \emph{loose edge} has no endpoints.  (Half and loose edges have a minor role in this paper; they appear only because of contraction.)  The set of loops and links is written $E^*$.  
Multiple edges are permitted.  We write $n := |V|$ and $V = \{v_1,v_2,\ldots,v_n\}$.  All our graphs have finite order and indeed are finite (except for root edges, when they appear).  
A \emph{(connected) component} of $\G$ is a maximal connected subgraph that is not a loose edge; we do not count a loose edge as a component.  
The number of components is $c(\G)$.  
For $S \subseteq E$, we denote by $c(S)$ the number of connected components of the spanning subgraph $(V,S)$ (which we call the \emph{components of $S$}) and by $\pi(S)$ the partition of $V$ into the vertex sets of the various components.  
We write $S_{vw}$ to denote any path in $S$ from $v$ to $w$ (if one exists).

\subsection{Gain graphs}\label{gaingraphs}

A \emph{gain graph} $\Phi = (\G,\phi)$ consists of a graph $\G=(V,E)$, a group $\fG$ called the \emph{gain group}, and an orientable function $\phi: E^* \to \fG$, called the \emph{gain mapping}.  
(Half edges do not have gains.  We may think of a loose edge as having gain $\1$.)  
The basic reference is \cite[Part I, Section 5]{BG}.  ``Orientability'' means that, if $e$ denotes an edge oriented in one direction and $e\inv$ the same edge with the opposite orientation, then $\phi(e\inv) = \phi(e)\inv$.  (Edges are undirected, i.e., they do not have fixed orientations.  We orient an edge only in order to state the value of its gain.)  
We sometimes use the simplified notations $e_{ij}$ for an edge with endpoints $v_i$ and $v_j$, oriented from $v_i$ to $v_j$, and $ge_{ij}$ for such an edge with gain $g$; that is, $\phi(ge_{ij})=g$.  (Thus $ge_{ij}$ is the same edge as $g\inv e_{ji}$.)  
The gain of a walk $e_1e_2\cdots e_l$ is $\phi(e_1e_2\cdots e_l) := \phi(e_1)\phi(e_2)\cdots\phi(e_l).$

In particular, an \emph{integral gain graph} has gain group $\fG=(\bbZ,+)$.

The notation $c(\Phi)$ means $c(\G)$.

An \emph{isomorphism} of two gain graphs $\Phi$ and $\Phi'$ with the same gain group $\fG$ is an isomorphism of underlying graphs that preserves edge gains.

A \emph{circle} is a connected 2-regular subgraph without half edges, or its edge set; for instance, a loop is a circle of length $1$.  We may write a circle $C$ as a word $e_1e_2\cdots e_l$; this means the edges are numbered consecutively around $C$ and oriented in a consistent direction.  
The gain $\phi(C)$ is well defined up to conjugation and inversion, and in particular it is well defined whether the gain is the identity $\1$ or not.
An edge set or subgraph is called \emph{balanced} if every circle in it has gain $\1$ and it has no half edges.  
If $B \subseteq E$ is balanced, the gain $\phi(B_{vw})$ of a path $B_{vw}$ is the same for every such path (if it exists).

An ordinary graph can be thought of as a gain graph where all edges have identity gain; thus, every circle is balanced.

For $W\subseteq V$, the subgraph \emph{induced} by $W$ is notated $\G\ind W$, or with gains $\Phi\ind W$.  The vertex set of $\G\ind W$ is $W$ and the edge set consists of all edges that have at least one endpoint in $W$ and no endpoint outside $W$; thus, half edges at vertices of $W$ are included, but loose edges are not.  
If $S \subseteq E$, then the subset of $S$ induced by $W$ is $S\ind W := E((V,S)\ind W)$.

\emph{Switching} $\Phi$ by a \emph{switching function} $\eta : V \to \fG : v_i \mapsto \eta(v_i)$ means replacing $\phi$ by 
$$\phi^\eta(e_{ij}) := \eta(v_i)\inv\phi(e_{ij})\eta(v_j).$$  
We write $\Phi^\eta$ for the switched gain graph $(\G,\phi^\eta)$.  Switching functions form a group $\fG^V$ under componentwise multiplication.  
Switching is an action of the group $\fG^V$ on the set $\fG^E$ of gain functions on the underlying graph.

A \emph{switching isomorphism} of gain graphs $\Phi$ and $\Phi'$ with the same gain group $\fG$ is an isomorphism of underlying graphs that preserves gains up to switching; that is, it is an isomorphism of $\Phi'$ with some switching of $\Phi$.

Consider a balanced edge set $S$.  Let $S_{v_iv_j}$ denote a path in $S$ from $v_i$ to $v_j$, if one exists; the gain $\phi(S_{v_iv_j})$ is independent of the particular path because $S$ is balanced.  
There is a switching function $\eta$ such that $\phi^\eta\big|_S\equiv\1$ \cite[Section I.5]{BG}; it is determined by any one value in each component of $S$ through the formula 
\begin{equation}\label{E:etatransfer}
\eta(v_j) = \phi(S_{v_jv_i}) \eta(v_i).  
\end{equation}
We call $\eta$ a \emph{switching function for $S$}.   Any two different switching functions for $S$, $\eta$ and $\eta'$, are connected by the relation
\begin{equation}\label{E:etarelation}
\eta' = \eta \cdot \alpha_W
\end{equation}
for constants $\alpha_W \in \fG$, one for each $W \in \pi(S)$.  (In fact, $\alpha_W = \eta(v_i)\inv \eta'(v_i)$ for any $v_i \in W$; this is easy to deduce from \eqref{E:etatransfer}.)  
Thus, as long as the endpoints of an edge $e_{ij}$ are in the same component of $S$, $\phi^\eta(e_{ij}) = \1$.  
(If $e_{ij}$ has endpoints in distinct components of $S$, then $\phi^\eta(e_{ij}) = \eta(v_i)\inv \phi(e_{ij}) \eta(v_j)$ can be anything, since $\eta(v_i)$ and $\eta(v_j)$ are independently choosable elements of $\fG$.)

The operation of deleting an edge or a set of edges is obvious.   The notation for $\Phi$ with $E\setminus S$ deleted, called the \emph{restriction} of $\Phi$ to $S$, is $\Phi|S=(V,S,\phi|_S)$.
The number of components of $S$ that are balanced is $b(\Phi|S)$ or briefly $b(S)$ (recall that this counts isolated vertices but not loose edges); 
$\pib(S)=\pib(\Phi|S)$ is the set $\{W\in\pi(S) : (S\ind W) \text{ is balanced}\};$ 
$V_0(S)$ is the set of vertices that belong to no balanced component of $S$; and $V_\textb(S)$ denotes the set of vertices of balanced components, $V_\textb(S) = V \setm V_0(S)$.  

Contraction is not so obvious.  We take the definition from \cite{BG}.  
First, we describe how to contract a balanced edge set $S$.  
We first switch by $\eta$, any switching function for $S$; then we identify each set $W\in\pi(S)$ to a single vertex and delete $S$.  The notation is $\Phi^{\eta}/S$.  This contraction depends on the choice of $\eta$, so $\Phi^{\eta}/S$ is well defined only up to switching.  (Soon, however, we shall see how to single out a preferred switching function.)  By this definition, contracting a balanced loop or loose edge is the same as deleting it (but that is not true for unbalanced loops).

For a general subset $S$ we first delete the vertex set $V_0(S)$, then contract the remaining part of $S$, which is the union of all balanced components of $S$, and delete any remaining edges of $S$.  
Edges not in $S$ that have one or more endpoints in $V_0(S)$ lose those endpoints but remain in the graph, thus becoming half or loose edges.  
So, the contraction has $V(\Phi/S)=\pib(S)$ and $E(\Phi/S) = E \setminus S$.

A balanced edge set $S$ is called \emph{closed} if any edge $ge_{ij}$ whose endpoints are joined by an open path $P \subseteq S$ with the same gain is itself in $S$.  (For a balanced loop $0e_{ii}$, the path $P$ has length 0 and the edge is necessarily in $S$.)  This is equivalent to saying $S$ equals its own closure; the closure of a balanced edge set $S$ is given by
\begin{equation}\label{E:bal-closure}
\clos(S) := S \cup \{ e\notin S : e \text{ is contained in a balanced circle } C \subseteq S \cup \{e\} \} .
\end{equation}
By \cite[Proposition I.3.1]{BG}, the closure of a balanced set is again balanced.  
The semilattice of all closed, balanced edge sets in $\Phi$ is written $\Latb \Phi$.

\subsection{Lattice-ordered gain groups}\label{orderedgains}

For the rest of this article we assume the gain group is lattice ordered.  
(It may be totally ordered; that case has some special features.)

We continue thinking of a balanced edge set $S$.  
The ordering singles out a particular switching function for $S$, the one for which the meet of its values on each set $W\in\pi(S)$ is the identity.  We call this the \emph{top switching function} and we write it $\eta_S$; it is what we use for switching throughout the rest of this article.  Its existence and uniqueness are proved in the following lemma.  

\begin{lem}\label{L:top}
The top switching function $\eta_S$ has the formula 
\begin{equation*}\label{E:etaS}
\eta_S(v) = \bigvee_{w\in W}  \phi(S_{vw}) ,
\end{equation*}
where $W\in\pi(S)$ is the part that contains $v$, and for its inverse 
\begin{equation*}\label{E:etaSinverse}
\eta_S(v)\inv = \bigwedge_{w\in W}  \phi(S_{wv}) .
\end{equation*}
\end{lem}  

\begin{proof}
Because $S$ is balanced, the gain of every path $S_{vw}$ is the same.  That is why the right side of \eqref{E:etaS} is well defined.

We use the identity $(\alpha \wedge \beta)\inv = \alpha\inv \vee \beta\inv$.

We know two properties of $\eta_S$.  As a switching function for $S$ it satisfies \eqref{E:etatransfer}.  As a top switching function it satisfies 
$
\bigwedge_{w \in W}  \eta_S(w) = \1
$
for every $W\in\pi(S)$.  Equation \eqref{E:etatransfer} lets us rewrite this as
$$
\bigwedge_{w \in W}   \phi(S_{wv}) \eta_S(v) = \1.
$$
Factoring out $\eta_S(v)$, 
$$
\eta_S(v) = \Big[  \bigwedge_{w \in W}   \phi(S_{wv})  \Big]\inv = \bigvee \phi(S_{vw}) .  
\qedhere
$$
\end{proof}

In view of the importance of the meet of switching-function values, we define 
$$
\eta(X) := \bigwedge_{v \in X} \eta(v)
$$
for $X \subseteq V$.

Now, to contract $S$ we first switch by $\eta_S$; then we identify each set $W\in\pi(S)$ to a single vertex and delete $S$.  The contraction $\Phi^{\eta_S}/S$, which we call the \emph{top contraction}, we usually write $\Phi/S$ for brevity.  
The contraction $\Phi/S$ is now a unique gain graph, because the gain-group ordering allows us to specify the switching function uniquely.

When the group is totally ordered, there is a \emph{top vertex} in every component of $S$, a vertex $t$ such that no path in $S$ that ends at $t$ has gain $< \1$; this is any vertex for which $\eta_S(t) = \1$.  Then the rule for defining $\eta_S$ is that its minimum value on the vertices of each component of $S$ is the identity. 
A top vertex may also happen to exist when $\fG$ is not totally ordered.  
If $t_i$ denotes a top vertex in the same component of $S$ as $v_i$, then the top switching function has the formula 
\begin{equation}\label{E:topvertexsw}
\eta_S(v_i) = \phi(S_{v_it_i}) .  
\end{equation}
The gain function $\phi$ switched by $\eta_S$ is given by the formula 
\begin{equation}\label{E:phiS}
 \phi^{\eta_S}(e_{ij}) =  \phi(S_{v_it_i})\inv \phi(e_{ij}) \phi(S_{v_jt_j}) = \phi(S_{t_iv_i}e_{ij}S_{v_jt_j}) .
\end{equation}
%

\subsection{Weights}\label{weights}

Suppose we have an abelian semigroup $\fW$ (written additively) and a group $\fG$.  We say $\fG$ \emph{acts on} $\fW$ if each $g \in \fG$ has a right action on $\fW$ which is a semigroup automorphism satisfying the usual identities, i.e., $g$ is a bijection $\fW\to\fW$ satisfying $(h+h')g = hg+h'g$, $0g = 0$, $(hg)g' = h(gg')$, and $h\1 = h$.

A \emph{weighted gain graph} $\Phih$ is a gain graph $\Phi$ together with a \emph{weight} function $h : V \to \fW$, where the gain group $\fG$ has a right action on the weight semigroup $\fW$.  We usually write $h_i := h(v_i)$.  The way $h$ transforms under switching is that 
$$
h^\eta_i = h^\eta(v_i) := h_i \eta(v_i)
$$ 
(as if $h_i$ were the gain of an edge oriented into $v_i$ from an extra vertex at which $\eta$ is the identity).  Thus, the switching group $\fG^V$ has a right action on the set $\fW^V$ of weight functions.  
Switching $\Phih$ means switching both $\Phi$ and $h$, i.e., $$\Phih^\eta := (\Phi^\eta, h^\eta).$$  

A \emph{switching isomorphism} of weighted gain graphs $\Phih$ and $(\Phi',h')$ with the same gain group $\fG$ is an isomorphism of underlying graphs that preserves gains and weights up to switching.  That is, it is an isomorphism of $(\Phi',h')$ with some switching of $\Phih$.

The restriction of $\Phih$ to an edge set $S$ is denoted by $\Phih|S := (\Phi|S,h)$.

The contraction rule is that, first, we always contract with top switching; and if $W\in\pib(S)$, then the weight function $h/S$ in the contraction $\Phih/S$ is given by 
\begin{equation}
h/S(W) := \sum_{v_i\in W} h_i^{\eta_S} = \sum_{v_i\in W} h(v_i)\eta_S(v_i) .
\label{E:wtcontract}
\end{equation}
If $R\subseteq S$ and $\pib(R)=\pib(S)$, then $h/R=h/S$.

If there is a top vertex $t_i$ in the component $S\ind W$ that contains $v_i$, then 
$
h_i^{\eta_S} = h_i \phi(S_{v_it_i})
$ 
and 
$
h/S(W) = \sum_{v_i \in W} h_i \phi(S_{v_it_i}) .
$

 We have occasion to contract the induction $\Phih\ind W := (\Phi\ind W,h|_W)$ of the entire weighted graph by the induction $S\ind W$ of an edge set, where $W\in\pib(S)$.
 
The next result states the fundamental properties of deletion and contraction of weighted gain graphs.

\begin{prop}\label{P:repeatedops}
In a weighted gain graph $\Phih$, let $S\subseteq E$ be the disjoint union of $Q$ and $R$.  Then 
\begin{gather*}
(\Phih/Q)/R = \Phih/S,  \\
(\Phih/Q)\setminus R = (\Phih\setminus R)/Q, \\ 
(\Phih\setminus Q)\setminus R = \Phih\setminus S.
\end{gather*}
\end{prop}

\begin{proof}
We may suppose $\Phi$ is connected.  The two latter formulas are obvious.  

The first one is not; indeed, in a purely technical sense it is false, since $V(\Phi/S) = \pib(\Phi|S)$ while $V((\Phi/Q)/R) = \pib(\Phi/Q|R)$; but it is correct if we identify $W \in \pib(\Phi|S)$ with $W''\in\pib(\Phi/Q|R)$ in the natural way:  $W$ corresponds to $W'' = \{ X\in\pib(\Phi|Q) : X\subseteq W \}$ and conversely $W''$ corresponds to $W = \bigcup W'' = \{ w\in V(\Phi) : w \in X \text{ for some } X \in W'' \}$.

In proving the first formula, the first step is to show that we can assume $S$ is balanced.  
It is a routine check to see that $\Phi/S$ and $\Phi/Q/R$ have the same half and loose edges; thus we may assume for the rest of the proof that $\Phi$ has no edges of those kinds.  

Since $V(\Phi/S) = \pib(S)$, we have
$$
\Phi/S = [ (\Phi\ind U)/(S\ind U) ] \ \text{ and } \ \Phi/Q/R = \big( (\Phi\ind U)/(Q\ind U) \big) / (R\ind U)  
$$
where $U := V_\textb(S)$.  
Since $S\ind U$ is balanced, both $Q\ind U$ and $R\ind U$ are also balanced. 
The weights of the contractions only appear on vertices of $\Phi/S$ so they depend only on vertices and edges in $U$; the same holds true for $\Phi/Q/R$.  It follows that
$$
\Phih/S = (\Phih\ind U)/(S\ind U) 
$$
and
$$
\Phih/Q/R = \big( (\Phih\ind U)/(Q\ind U) \big) / (R\ind U) .
$$
This proves that we may confine our attention to the balanced spanning subgraph $(U,S\ind U)$ in $\Phi\ind U$; thus, we may from now on assume $S$ is balanced.

Let $\eta'_S$ be the top switching function for $\Phi^{\eta_Q}|S$ and let $\eta''_R$ be that for $\Phi/Q|R$.  (Recall that $\Phi/Q$ means $\Phi^{\eta_Q}/Q$.)  
The key to the proof is the factorization identity
\begin{equation}\label{E:etacomp}
\eta_S(v) = \eta_Q(v) \eta'_S(v) ,
\end{equation}
which shows that the effect of $\eta_S$, which is to switch so $\phi|_S$ becomes $\1$, can be divided into two stages: first switching by $\eta_Q$ so that $\phi|_Q$ becomes $\1$, and then switching by $\eta'_S$, which is constant on components of $Q$.  

In proving \eqref{E:etacomp}, first we compare $\eta_Q$ and $\eta_S$.  Since they are two switching functions for $Q$, they are related by \eqref{E:etarelation}.  
Specifically, let $X \in \pi(Q)$ and $W \in \pi(S)$, with $X \subseteq W$; then $\eta_S(v) = \eta_Q(v) \alpha_X$ for $v \in X$.  Taking the meet over $X$, $\eta_S(X) = \eta_Q(X) \alpha_X = \1 \alpha_X$, so 
$$
\eta_S(v) = \eta_Q(v) \eta_S(X) 
$$
for $v \in X$.  
Next we show that $\eta_S(X) = \eta'_S(v)$.  Define $\bar\eta := \eta_Q \eta'_S.$  It is easy to verify that $\bar\eta$ satisfies \eqref{E:etatransfer} so it is a switching function for $S$.  Taking the meet over all $v \in W$, and taking note that $\eta'_S(v) = \eta'_S(X)$ for $v \in X$ because $\eta'_S$ is constant on $X$, we find that
\begin{align*}
\bar\eta(W) &= \bigwedge_{v\in W} \eta_Q(v) \eta'_S(v) 
 = \bigwedge_{X \in W''} \Big[ \bigwedge_{v \in X} \eta_Q(v) \Big] \eta'_S(X) \\
 &= \bigwedge_{X \in W''} \1 \eta'_S(X) 
 = \bigwedge_{v\in W} \eta'_S(v) = \1 .
\end{align*}
Thus $\bar\eta$ is a top switching function for $S$ and, as there is only one, it equals $\eta_S$.  This proves \eqref{E:etacomp}.

From \eqref{E:etacomp} it follows that $\phi^{\eta_S} = (\phi^{\eta_Q})^{\eta'_S}$ and (by the group action on weights) $h^{\eta_S}(v) = \left( h^{\eta_Q} \right)^{\eta'_S}(v)$, thus establishing that
 \begin{equation} \label{E:tvcompose}
 \Phih^{\eta_S} = \big( \Phih^{\eta_Q} \big) ^{\eta'_S} .
 \end{equation}

Now we can analyze the process of contraction.  
We know from \cite[Theorem I.4.7 and the proof of Theorem I.5.4]{BG} that $\Phi/Q/R \sim \Phi/S$ (where $\Phi_1 \sim \Phi_2$ means that each of them is a switching of the other).  
But the switching equivalence is really equality because the switching functions employed are related by Equation \eqref{E:etacomp}.  Thus, $\Phi/S = \Phi/Q/R$.  

The last step is to prove that weights contract properly.  The key here is that contraction by $Q$ commutes with two-stage switching, i.e., 
\begin{equation} \label{E:tvcontract}
( h^{\eta_Q}/Q ) ^{\eta''_R} = ( h^{\eta_Q} ) ^{\eta'_S} / Q .
\end{equation}
Observe that $\eta'_S$ is constant on each $X \in \pi(Q)$ and its common value is $\eta''_R(X)$ (the value of $\eta''_R$ on $X$ as a vertex in $\Phi/Q$).  
Expanding both sides according to the definitions of switching and contraction, this is equivalent to 
$$
\Big( \sum_{w \in X} h^{\eta_Q}(w) \Big) \eta''_R(X) = \sum_{w \in X} h^{\eta_Q}(w)\eta'_S(w) ,
$$
which is true because $\eta'_S(w) = \eta''_R(X)$.  That concludes the proof of \eqref{E:tvcontract}.

Equations \eqref{E:tvcompose} and \eqref{E:tvcontract} imply the double contraction formula through the sequence of transformations 
\begin{align*}
 \Phih/Q/R &= \big( \Phih^{\eta_Q}/Q \big) ^{\eta''_R} / R
 &\quad\text{ by definition} \\
	&= \big( \Phih^{\eta_Q \eta'_S}/Q \big) /R 
 &\quad\text{ by \eqref{E:tvcontract}} \\
	&=  \big( \Phih^{\eta_S}/Q \big) /R            
 &\quad\text{ by \eqref{E:tvcompose}} \\
	&= \Phih^{\eta_S} / S
 \end{align*}
in the loose sense previously defined in terms of the correspondence $W \leftrightarrow W''$.
 \end{proof}

\myexam{Weighted integral gain graphs; linearly ordered group weights}\label{X:wigg}

Our original example \cite{SOA} was that of \emph{weighted integral gain graphs}, where the gain group is the additive group of integers and the weight semigroup is the integers with the operation of maximization.  In other words, $\fG = (\bbZ, +)$ and $\fW = (\bbZ,\max)$.  
The gains act on the weights by translation, i.e., addition.

A similar kind of example exists for every linearly ordered gain group, with $\fW = (\fG,\max)$ or $(\fG,\min)$.  
\end{exam}

\myexam{Weights and gains in an integer lattice}\label{X:Zd}

A more general example, of which Examples \ref{X:Qorder2i} and \ref{X:Qorder2} are a specimen, has the gain group and the weight semigroup both equal to $\bbZ^d$, the $d$-dimensional integer lattice in $\bbR^d$.  The gain group is $(\bbZ^d,+)$ with the lattice operations $\vee$ and $\wedge$, meaning componentwise $\max$ and $\min$.  The weight semigroup is $(\bbZ^d,\vee)$.  The action of the group on the semigroup is by translation, i.e., $\bx\bw := \bx+\bw$ for $\bx\in(\bbZ^d,+)$ and $\bw\in(\bbZ^d,\vee)$.  
\end{exam}

\myexam{Semilattice weights}\label{X:setwts}

To further generalize Example \ref{X:wigg}, let $\fW$ be a semilattice with a $\fG$-action; the semigroup operation is the semilattice operation.  
In an important example of this type there is a set $\fC$ on which there is a right action of the gain group;
the weights are subsets of $\fC$, i.e., $\fW \subseteq \cP(\fC)$; and the semigroup operation is set intersection---so $\fW$ must be closed under intersection.  (In Section \ref{coloration} $\fC$ will be a color set and the weight $h_i \subseteq \fC$ will be treated as the list of colors possible for vertex $v_i$.)

Generalizing minimization, let $\fC$ be a partially ordered set and let the weights be order ideals in $\fC$.  If $\fC$ is a meet semilattice, one may restrict the weights to be principal ideals.  
There are also the order duals of these examples.

When, on the other hand, $\fC=\fG$ with the right translation action, one may take $\fW = \cP(\fG)$, for instance, or the class of principal dual order ideals (since $\fG$ is a lattice), or the class of sets that have a lower bound (that is, all subsets of principal dual ideals).  
The dual of this last, with gain group $\bbZ$ was the prototype of weighted gain graphs, as we explain next.  
\end{exam}

\subsection{Rooted integral gain graphs are weighted}\label{rigg}

Since the inspiration for this article was a reconsideration of \cite{SOA}, abetted by \cite{NW}, we wish to explain to readers of \cite{SOA} how its rooted integral gain graphs are a notational variant of certain weighted integral gain graphs.

A \emph{rooted integral gain graph} is an integral gain graph $\Psi$ with a root vertex $v_0$ adjacent to all other vertices in such a way that the gains of edges $e_{0i}$ form an interval $(-\infty,h_i]$ in the gain group $\bbZ$, the infinite cyclic group.  In \cite{SOA} we studied the chromatic function of a rooted integral gain graph, which counts proper colorations in an interval $(-\infty,m]$, as explained in Example \ref{X:intervalcol}.

An equivalent presentation replaces the root by an integral weight $h_i$ on each vertex of $\Phi := \Psi \setminus v_0$.  Switching of $\Psi$ as in \cite{SOA}, when transferred to $\Phi$, implies the rule $h^\eta = h+\eta$ for switching the weights.  (We write a plus sign here because the action of $\fG=\bbZ$ on $\fW=\bbZ$ in this example is by addition.)  That is the rule adopted and generalized in Section \ref{weights}.

Similarly, the standard gain-graphic contraction on $\Psi$, when reinterpreted in terms of integral weights $h_i$ on $\Phi$, assigns to a set $W \in V(\Phi/S)$ a weight equal to the maximum weight of a vertex in $W$ after switching; thus the weight set is $\bbZ$ with the operation of maximization.  

Trying to generalize this equivalence to more arbitrary root-edge gain sets, and reading about the weight operation in \cite{NW}, we realized that the weight operation is a semigroup operation, different from the gain-group operation, and the weight semigroup is separate from the gain group; and so we settled upon the generalization in this article.

\section{A Tutte-invariant polynomial}\label{invariant}

A function $f$ defined on weighted gain graphs (with fixed gain group and weight semigroup) is a \emph{Tutte invariant} if it satisfies the four conditions from the introduction:
\begin{enumerate}

\item[(Ti)] (Additivity)  For every link $e$, 
$$
 f(\Phi,h) = f(\Phi\setminus e,h) + f(\Phi/e,h/e) ,
$$
where $h/e$ denotes the contracted weight function.

\item[(Tii)] (Multiplicativity)  The value of $f$ on $\Phih$ is the product of its values on the components of $\Phih$.

\item[(Tiii)] (Isomorphism Invariance)  If $\Phih$ and $(\Phi',h')$ are isomorphic, then $f(\Phi,h) = f(\Phi',h')$.  

\item[(Tiv)] (Unitarity)  $f(\eset) = 1$.

\end{enumerate}

We present here an algebraic Tutte invariant.  We need variables $u_k$ for all $k\in\fW$; the collection of all $u_k$'s is denoted by $\bu$.  The \emph{total dichromatic polynomial} of a weighted gain graph is
\begin{equation}\label{E:dichromaticpoly}
 Q_\Phih(\bu,v,z) :=
 \sum_{S\subseteq E} v^{|S|-n+b(S)} z^{c(S)-b(S)} 
 \prod_{W\in\pib(S)} u_{h/S(W)} ,
\end{equation}
where 
$$h/S(W) := \sum_{w \in W} h^{\eta_S}(w).$$  

The important values of $z$ are 0 and 1.  When $z=0$ we have a sum only over balanced sets $S$; this is the \emph{balanced dichromatic polynomial} $Q^\textb_\Phih(\bu,v)$.  When $z=1$ we sum over all edge sets $S$; this gives the \emph{(ordinary) dichromatic polynomial} $Q_\Phih(\bu,v)$.  
These specializations refine the balanced and ordinary dichromatic polynomials of a gain graph or biased graph \cite[Section III.3]{BG}, which are obtained by setting all $u_k=u$ and $z=0$ or $z=1$.  Thus the total polynomial with all $u_k=u$ fills a gap in the theory of \cite[Part III]{BG} by unifying the balanced and ordinary polynomials.  The advantage of the total polynomial is in having one Theorem \ref{T:dandc} instead of two, separately for the balanced and ordinary polynomials.

For a graph with no edges,
\begin{equation}\label{E:initial}
Q_{((V,\eset),h)} = \prod_{v_i \in V} u_{h_i} .
\end{equation}
If $e$ is a balanced loop or a loose edge,
\begin{equation}\label{E:loop}
Q_\Phih = (v+1) Q_{(\Phi\setminus e,h)} .
\end{equation}

\begin{thm}\label{T:dandc}
The total dichromatic polynomial $Q_\Phih(\bu,v,z)$ is a Tutte invariant of weight\-ed gain graphs. 
\end{thm}

\begin{proof}
Unitarity, multiplicativity, and isomorphism invariance are clear.  
For additivity we follow the usual proof method, dividing up the terms of the defining sum into two parts: those sets $S$ that do not contain the link $e$ and those sets that do contain $e$.  The sum of the former terms equals $Q_{\Phih\setm e}(\bu,v,z)$ and the sum of the latter, we shall see, equals $Q_{\Phih/e}(\bu,v,z)$. 

A set $S\ni e$ contracts to a set $R=S\setminus e$ in $\Phi/e$ whose balanced components correspond to those of $S$.  That is, if $S_0$ is a balanced component of $S$, then $S_0$ (if $e \notin S_0$) or $S_0/e$ (if $e \in S_0$) is a balanced component of $R$, and vice versa.  (This follows from \cite[Lemma I.4.3]{BG}.)  So $b(\Phi/e|R)=b(\Phi|S)$.  
Since $\Phi/e$ has order $n-1$, the term of $S$ in $Q_\Phih$ and that of $R$ in $Q_{\Phih/e}$ are the same except for the factors $u_{h/S(W)}$ in the former and $u_{h/R(W'')}$ in the latter, where $W''$ is the member of $\pib(\Phi/e|R)$ that corresponds to $W$.  We want to show that these factors are equal, i.e., that $h/S(W) = h/R(W'')$.  But the former is $h/S$ and the latter is $(h/e)/R$, which we know by Proposition \ref{P:repeatedops} to be equal.

It follows that $Q$ satisfies additivity, so is a Tutte invariant.
\end{proof}

The total dichromatic polynomial does not, in general, behave predictably under switching.  One can see that by comparing the switched polynomial $Q_{\Phih^\eta}$ to the unswitched one, $Q_\Phih$.  We omit the calculation.

\begin{exam}[Example \ref{X:Qorder2i}, continued]\label{X:Qorder2}
Now we can derive the polynomial in Example \ref{X:Qorder2i}.  

First we explain the gains and weights fully: this example is a case of Example \ref{X:Zd}, where the gains add and the weights maximize.

An edge subset with at most one edge is balanced; $b(\eset)=2$, $b(\{e_i\})=1$ for any edge $e_i$.  A subset with two or three edges is unbalanced.  

We calculate the total dichromatic polynomial step by step:
\begin{equation*}\label{E:Qorder2}
\begin{aligned}
 Q_\Phih(\bu,v,z) :=\
 &\sum_{S\subseteq E} v^{|S|-n+b(S)} z^{c(S)-b(S)} 
 \prod_{W\in\pib(S)} u_{h/S(W)} \\
 =\ & v^{0-2+2} z^{2-2} u_{h/\eset(\{v_1\})} u_{h/\eset(\{v_2\})} 
     + \sum_{i=1}^3 v^{1-2+1} z^{1-1} u_{h/e_i(V)} \\
     &+ \sum_{j=1}^3 v^{2-2+0} z^{1-0} u_{h/(E\setm e_j)(V)} 
     + v^{3-2+0} z^{1-0} u_{h/E(V)} \\
 =\ & u_{(2,0)}u_{(-1,3)} + [u_{(2,3)} + u_{(4,3)} + u_{(2,3)}] + z [3 + v] .
\end{aligned}
\end{equation*}

Explanations:  $h/\eset$ is simply $h$; there is no contraction.  
The top switching function $\eta_{e_i}$ for edge $e_i=\bg_i(v_1,v_2)$ must satisfy $-\eta_{e_i}(v_1) + \bg_i + \eta_{e_i}(v_2) = (0,0)$ (because it is a switching function for $B=\{e_i\}$) and $\eta_{e_i}(v_1) \wedge \eta_{e_i}(v_2) = (0,0)$, the identity in the gain group.  That is, $[\eta_{e_i}(v_2)+\bg_i] \wedge \eta_{e_i}(v_2) = (0,0).$  
It follows that $\eta_{e_i}(v_1)=\bg_i^+$ and $\eta_{e_i}(v_2)=\bg_i^-$ (recall the positive and negative parts of a vector from the beginning of Section \ref{defs}).
The values of $\eta_{e_i}(v_j)$ are given in Table \ref{Tb:Qorder2}.  
The weight of the single vertex $V$ of the contraction by $e_i$, from Equation \ref{E:wtcontract}, is 
\begin{align*}
h/e_i(V) &= h^{\eta_{e_i}}(v_1) \vee h^{\eta_{e_i}}(v_2) = h(v_1)\eta_{e_i}(v_1) \vee h(v_2)\eta_{e_i}(v_2)] \\
&= [h(v_1)+\eta_{e_i}(v_1)]\vee[h(v_2)+\eta_{e_i}(v_2)]
\intertext{(the switching values, which belong to the gain group, act on weights by translation)}
&= [(2,0)+\eta_{e_i}(v_1)]\vee[(-1,3)+\eta_{e_i}(v_2)] .
\end{align*}
The values, obtained by componentwise maximization ($\vee$), are in Table \ref{Tb:Qorder2}.  
Contracting $\Phi$ by two or three edges, thus by an unbalanced edge set $S$, gives a graph with no vertices; hence there are no weights associated to the term of $S$, so $\prod_{W\in\pib(S)} u_{h/S(W)} = 1$.
\begin{table}[ht]
\begin{tabular}{|c||c|c|c|c|}
\hline
$B=$		&$\eset$	&$\{e_1\}$&$\{e_2\}$&$\{e_3\}$
\\
$W=$	&$\{v_1\},\{v_2\}$	&$V$	&$V$	&$V$
\\\hline\\[-13pt]
$\eta_B(v_1)$	&$(0,0)$	&$(0,0)$	&$(2,0)$	&$(0,2)$
\\
$\eta_B(v_2)$	&$(0,0)$	&$(0,0)$	&$(0,0)$	&$(1,0)$
\\\hline\\[-13pt]
$h^{\eta_B}(v_1)$	&$(2,0)$	&$(2,0)$	&$(4,0)$	&$(2,2)$
\\
$h^{\eta_B}(v_2)$	&$(-1,3)$	&$(-1,3)$	&$(-1,3)$	&$(0,3)$
\\\hline\\[-13pt]
Wt.	&$(2,0),(-1,3)$	&$(2,3)$	&$(4,3)$	&$(2,3)$
\\\hline
\end{tabular}
\medskip
\caption{The top switching function $\eta_B$, the switched weights $h^{\eta_B}(v_i)$, and the contracted weights (``Wt.'') for Example \ref{X:Qorder2}.}
\label{Tb:Qorder2}
\end{table}
\end{exam}

 
\section{Forest expansion}\label{tree}

We turn to an expression for the balanced dichromatic polynomial, $Q_\Phih(\bu,v,0)$, that depends on a linear ordering of the edge set.  We fix one such ordering $O$ and in terms of it we define a spanning-forest expansion similar to the Tutte polynomial of a matroid.  (We did not find an expansion in terms of spanning trees.)  The details are in Section \ref{treeexpansion}, after some preliminary work with independent sets in semimatroids and gain graphs.

\subsection{Activities in semimatroids}\label{semi}

A semimatroid is a generalization of a matroid that extends properties like rank and closure of the family of balanced edge sets in a gain graph.  A theory was created by Wachs and Walker in \cite{WW} in terms of the ``geometric semilattice'' of closed sets and later Ardila (in \cite{ArdilaPhD, ArdilaTPA}, and especially \cite{ArdilaSemi}) developed a comprehensive treatment in terms of subsets of a set, in analogy to matroid theory.  Our treatment of semimatroids was developed without knowing Ardila's.  The two necessarily overlap as they concern the same objects, but the overlap is partial; so in order to be self-contained, we present our treatment as we need it for this paper.  We begin with general semimatroids and then specialize to those of gain graphs.

Just as with matroids, there are many equivalent ways to define a semimatroid.  
We define a semimatroid in terms of a matroid $M_0$ with ground set $E_0$ and a basepoint $e_0$ that is not a loop.  A subset of $E := E_0 \setm e_0$ whose closure in $M_0$ does not contain $e_0$ is called \emph{balanced}; the family of balanced sets is denoted by $\cPb(M)$.  The \emph{semimatroid} $M$ associated with $(M_0,e_0)$ is the family of all balanced subsets of $E$ (not all subsets, unless $E$ is balanced) with closure operator, rank function, closed or independent sets, circuits, and so forth the same as those of $M_0$ but restricted to balanced sets.      
The ground set of $M$ is $E$.  

For instance, the independent sets of $M$ are the ones of $M_0$ whose closures do not contain $e_0$.  The closed sets of $M$ are those of $M_0$ that do not contain $e_0$.  The rank $\rk_M(S)$ of a subset $S \subseteq E$ is defined only if $S$ is balanced; then it equals $\rk_{M_0}(S)$.

(A matroid, by contrast, is defined on the family of all subsets of its ground set.  Matroids are semimatroids, but in general semimatroids are not matroids.  Specifically, if $e_0$ is a coloop in $M_0$, then $M$ is the matroid $M_0 \setm e_0$.  Otherwise, $E$ is not balanced; then $M$ is technically not a matroid because its whole ground set $E$ is not in $\cPb(M)$.)

The restriction of $M_0$ or $M$ to a subset $S$ of the ground set is denoted by $M_0|S$ or $M|S$, respectively.  Thus $M|S$, if $S \subseteq E$, is the family of balanced sets of $M|(S \cup e_0)$ with the rank function, balanced independent sets, et al., restricted to $S$.

A fundamental fact is that if $S$ is balanced, the closure $\clos_0 S$ (in $M_0$) is balanced.  Consequently, the closure $\clos S$ in $M$ equals $\clos_0 S$.   Also, any circuit in $\clos S$ is balanced.  A maximal balanced independent set, that is, a maximal independent set of $M$, is called a \emph{semibasis}.  It lacks one element to be a basis of $M_0$; thus, $\rk M = \rk M_0 - 1$.

We develop some facts about independent sets, activities, and broken circuits in a semimatroid.  Some of them are already known for matroids.  We could not find an explicit source for exactly these results, but \cite{ELV} and \cite[Section 2]{KRS} have theorems along similar lines.  A reference for the fundamentals of activities in matroids is any of \cite{Bjorner, DCG, BOTutte}.

First, some basic definitions.  Let $F$ be independent in $M_0$.  
For a point $e \in (\clos_0 F) \setm F$, there is a unique circuit contained in $F\cup e$; it is called the \emph{fundamental circuit} of $e$ with respect to $F$ and denoted by $C_F(e)$.  It is balanced if (but not only if) $F$ is balanced.  
For a point $f \in F$, we call $\clos_0(F) \setm \clos_0(F \setm f)$ the \emph{fundamental relative cocircuit} of $f$ with respect to $F$, written $D_F(f)$.  
(By a \emph{relative cocircuit} we mean a cocircuit in $M_0|\clos_0 (F)$; it need not be a cocircuit in $M$.)  
If $F$ is balanced, the closures are in $M$ so $\clos_0$ can be replaced by $\clos$.

We fix a linear ordering $O$ of $E$ and extend it to $E_0$ arbitrarily.  
Consider an independent set $F$ of the semimatroid $M$ (that is, a balanced independent set of $M_0$).  
We say that a point $e \in E$ is \emph{externally active} (in $M$) with respect to $F$ if $e \notin F$ and $e$ is the largest point in $C_F(e)$ (so only a point in $(\clos F) \setm F$ can be externally active).  
A point $e$ is \emph{internally active} (in $M$) with respect to $F$ if it is in $F$ and it is the largest point in $D_F(e)$.  
A point that is not active is \emph{internally inactive} if it belongs to $F$ and \emph{externally inactive} if it belongs to $(\clos F) \setm F$.  The sets of internally or externally active or inactive points with respect to $F$ are denoted by $\IA(F)$, $\EA(F)$, $\II(F)$, $\EI(F)$.  
The number of externally active points is $\epsilon(F)$.  The number of internally active points is $\iota(F)$.  

The definitions for $M_0$ are the same, except for the omission of the word ``balanced'' and replacement of $\clos$ by $\clos_0$.  Thus we have $\IA_0(F)$, et al.; but when $F$ is balanced, these are the same as $\IA(F)$.

A \emph{broken (balanced) circuit} is a (balanced) circuit with its largest element removed.  (A broken circuit $C\setm e$ is balanced if and only if its circuit $C$ is balanced, because $\clos_0(C\setm e) = \clos_0 C$; thus, a balanced broken circuit is the same as a broken balanced circuit.)
 
\begin{lem}\label{L:bbcmatroid}
Let $F$ be independent in the semimatroid $M$.  Then $\II(F)$ contains every broken balanced circuit in $F$.
\end{lem}

\begin{proof}
If $D$ is a broken balanced circuit in $F$, there is a point $e \in (\clos F) \setm F$ which is maximal in its fundamental circuit $C_F(e) = D \cup \{e\}$.  Any $f \in D$ is internally inactive because $e >_O f$, both $e,f \in D_F(f)$, and $e,f \in (\clos F) \setm \clos(F\setm f)$.  Thus, $D \subseteq \II(F)$.
\end{proof}

\begin{lem}\label{L:iicompare}
If $F'$ is independent in the matroid $M_0$ and $F \subseteq F'$, then $\II_0(F) \subseteq \II_0(F')$.
\end{lem}

\begin{proof}
Let $f \in \II_0(F)$.  Thus, there exists $e >f$ in $\clos_0(F) \setm \clos_0(F\setm f)$.  
The fundamental circuit of $e$ with respect to $F$ is contained in $F \cup e$ but not in $F\cup e \setm f$; therefore it is not contained in $F' \cup e \setm f$, so $e \notin \clos_0(F'\setm f)$.  
Since $e \in \clos_0(F') \setm \clos_0(F'\setm f)$, $f$ is internally inactive in $F'$.
\end{proof}

Each point set $S \subseteq E_0$ has a \emph{minimal basis} $F(S)$, which is the basis of the restriction $M_0|S$ that is lexicographically first according to $O$; it is the one obtained by the greedy algorithm applied to $S$.  It is balanced if and only if $S$ is balanced, because $\clos _0(S)=\clos _0(F(S))$.  
The next lemma says that the inverse of the mapping $S \mapsto F(S)$ partitions the power set of $E_0$ into intervals $[F, F\cup\EA_0(F)]$, one for each independent set $F$, and either all sets in the interval are balanced or all are unbalanced.  (This partition is an analog for semimatroids of one due to Crapo for matroids \cite{CrapoTP}, also in Bj\"orner \cite[Prop.\ 7.3.6]{Bjorner}.)

\begin{lem}\label{L:basismatroid}
Let $F$ be independent in $M_0$ and let $S \subseteq E_0$.  For the minimal basis of $S$ to be $F$, it is necessary and sufficient that $F \subseteq S \subseteq F \cup \EA_0(F)$.  Further, $F \cup \EA_0(F)$ is balanced if and only if $F$ is balanced.
\end{lem}

\begin{proof}
Assume $F$ is the minimal basis of $S$ and write $F = e_1e_2 \cdots$ in increasing order in $O$.  Every $e \in S \setm F$ has a fundamental circuit $C_F(e)$.  Suppose $e$ is not externally active with respect to $F$, so that $C_F(e) = \cdots e e' \cdots$; let $e' = e_{k+1}$.  
The set $\{e_1,\ldots,e_k,e\}$ is independent because the only circuit it could contain is $C_F(e)$, but $e' \in C_F(e) \setm\{e_1,\ldots,e_k,e\}$.  
Consider the greedy algorithm for finding $F$.  After choosing $e_1,\ldots,e_k$, the next point chosen cannot be $e'$, because $e$ (or some other point different from $e'$) would be preferred as it has not been chosen, it precedes $e'$ in the ordering, and $\{e_1,\ldots,e_k,e\}$ is independent.  Thus, $e_{k+1} \neq e'$.  This is a contradiction.  Therefore, $e$ must be externally active.

Assume $F \subseteq S \subseteq F \cup \EA_0(F)$.  Thus, $F$ is a basis for $S$; we want to show it is minimal.  
Let $e \in S \setm F$ and write $C_F(e) = e_1\cdots e_k e_{k+1}$ in the ordering $O$; then $e = e_{k+1}$.  In the greedy algorithm for constructing the minimal basis $F(S)$, each point $e_1, \ldots,e_k,e_{k+1}$ is considered in order for inclusion.  
Let $F_i(S)$ be the set of points that have already been chosen when $e_i$ is considered for inclusion.  
If $e_i$ is not then chosen for $F(S)$, it is because $e_i \in \clos F_i(S)$.  If $e_i$ is chosen, then $e_i \in F_{i+1}(S)$.  
Thus, all of $e_1,\ldots,e_k \in \clos F_{k+1}(S)$.  It follows that $e \in \clos F_{k+1}(S)$, so $e \notin F(S)$.  This shows that no point of $\EA_0(F)$ can belong to the minimal basis $F(S)$; hence, $F(S) \subseteq F$ and by comparing ranks we see that $F(S) = F$.

The last part of the lemma follows because $\EA_0(F) \subseteq \clos_0 F$, which is balanced if and only if $F$ is balanced.
\end{proof}

Suppose we already have a balanced independent set $F$ that we want to extend to a semibasis.  We do that by applying the \emph{reverse greedy algorithm}.  That means we take $E\setm F$ and scan down it from the largest point (in the ordering $O$) to the smallest, adding a point to the independent set whenever the resulting set remains independent and balanced.  The set obtained in this way, $T(F)$, is a semibasis that is lexicographically maximum among all semibases that contain $F$.

\begin{lem}\label{L:rgabasis}
For an independent set $F$ in $M$, $T(F)$ has the following properties:
\begin{enumerate}
\item[(o)] If $F$ is a semibasis then $T(F)=F$; in particular, $T(T(F)) = T(F)$.
\item[(i)] $F \supseteq \II(T(F))$ and $T(F) \setm F \subseteq \IA(T(F))$.
\item[(ii)] $\II(F) \subseteq \II(T(F))$.  
\item[(iii)] $\EA(F) \subseteq \clos(\II(T(F))$.
\item[(iv)] $\EA(F) = \EA(\II(T(F))) = \EA(T(F))$; thus, $\epsilon(F) = \epsilon(T(F))$.
\end{enumerate}
\end{lem} 

\begin{proof}
Part (o) is immediate from the definition.

In (i) the two statements are obviously equivalent; we prove the latter.  Suppose we have a balanced independent set $F'$ and a point $e \notin \clos F'$; call $e$ \emph{$F'$-tolerable} if $F' \cup \{e\}$ is balanced.  
Write $T(F) \setm F = e_k \cdots e_1$ in increasing order, so that each $e_i$ is the largest $F \cup \{e_1,\ldots,e_{i-1}\}$-tolerable point.  Since
$$
e_i \notin \clos(T(F)\setm e_i) \supseteq \clos(F\cup\{e_1,\ldots,e_{i-1}\}),
$$
$e_i$ is larger than any other $F\cup\{e_1,\ldots,e_{i-1}\}$-tolerable point not in $\clos(T(F)\setm e_i)$.  That is, it is internally active.

In Part (ii), $\II(F) \subseteq \II(T(F))$ by Lemma \ref{L:iicompare}.  

In (iii), $e$ is maximal in $C_F(e)$ for $e \in \EA(F)$.  By Lemma \ref{L:bbcmatroid} and Part (ii), $C_F(e) \setm \{e\} \subseteq \II(F) \subseteq \II(T(F))$.

For (iv), suppose we have two balanced independent sets, $F_1 \subseteq F_2$.  Obviously $\EA(F_1) \subseteq \EA(F_2)$, because $(\clos F_1) \setm F_1 \subseteq (\clos F_2) \setm F_2$.  If $\EA(F_2) \subseteq \clos F_1$, then $\EA(F_2) \subseteq \EA(F_1)$; thus the two $\EA$s are equal.  Now apply this fact to $F_1=\II(T(F))$ and $F_2 = F$ or $T(F)$, recalling (iii).
\end{proof}

\subsection{Activities in gain graphs}\label{ggactivities}

When we come to gain graphs, the semimatroid we need is that associated with the balanced edge sets of $\Phi$.  Here, $M_0$ is the \emph{complete lift matroid} $L_0(\Phi)$, which is the matroid on $E_0:=E(\Phi)\cup\{e_0\}$ with rank function for $S \subseteq E$ given by
\begin{align*}
\rk(S) &= \begin{cases}
 n-c(S) &\text{ if $S$ is balanced,} \\ n-c(S)+1 &\text{ if $S$ is unbalanced} ,
\end{cases} \\
\rk(S\cup e_0) &= n-c(S)+1
\end{align*}
\cite[Section II.4]{BG}.  In a way, the complete lift matroid generalizes the usual graphic matroid $G(\G)$, since when $\Phi$ is balanced, $e_0$ is a coloop and $G(\G) = L_0(\Phi) \setm e_0$.  We call the semimatroid associated with $L_0(\Phi)$ and $e_0$ the \emph{semimatroid of graph balance} of $\Phi$.

(For those concerned with loose and half edges: In this section we treat a loose edge $e$ as a balanced loop and a half edge as an unbalanced loop since in the matroid the two types behave exactly the same.)  

Here is how the previous discussion of semimatroids applies to gain graphs.  
A balanced circle is the same thing as a semimatroid circuit, i.e., it is a matroid circuit (in $L_0(\Phi)$) that is a balanced edge set.  
A spanning forest $F$ is the same as a balanced independent set; its closure defined in graphical terms is 
$$
\clos(F) := F\cup \{ e\notin F : F \cup \{e\} \text{ contains a balanced circle}\}.
$$
The reason is that $C_F(e)$, if it exists, must be balanced so it is a balanced circle; it is called the \emph{fundamental circle} of $e$ with respect to $F$.  
For an edge $e\in F$, in $F\setminus e$ one component of $F$ is divided into two; the fundamental relative cocircuit of $e$ with respect to $F$ is the set $D_F(e)$ of edges $f\in E$ that join these two into one (since $F\cup f$ is balanced).
To clarify these ideas we give a descriptive lemma that is particular to gain graphs.
Note that our definitions of activity, derived from the semimatroid of balanced edge sets, differ from the usual ones for graphs.  
A \emph{broken balanced circle} is a balanced circle with its largest edge removed.  

\begin{lem}\label{L:ea}
Suppose $\Phi$ is a gain graph with no balanced digons.  Let $F$ be a spanning forest in $\Phi$.  Then $\EA(F)$ is the set of all edges $e\notin F$ such that $e \in (\clos F) \setm F$ and $C_F(e)\setm e$ is a broken balanced circle.
\end{lem}

\begin{proof}
If an edge $e \notin F$ is externally active, it is in $(\clos F) \setm F$ and it is maximal in $C_F(e)$.  The latter implies that $C_F(e)\setm e$ is a broken balanced circle.  

To prove the converse, assume $C_F(e)$ exists and $D := C_F(e)\setm e$ is a broken balanced circle.  Either $e$ is maximal in $C_F(e)$, so $e$ is externally active, or $D \neq \eset$ and there is an edge $e' \notin F$, other than $e$, such that $C_F(e') \setm e' = D$.  Then $e$ and $e'$ are parallel links with the same endpoints.  Because they form a digon in $\clos F$, which is balanced since it is the closure of the balanced set $F$, they form a balanced digon, contrary to the assumption.  Consequently, $e'$ cannot exist.
\end{proof}

\subsection{The forest expansion}\label{treeexpansion}

The \emph{forest expansion} of $\Phih$ is
\begin{equation}\label{E:treepoly}
 F_{\Phih,O}(\bu,y) := \sum_F  y^{\epsilon(F)} \prod_{W\in\pi(F)} u_{h/F(W)} ,
\end{equation}
summed over all spanning forests $F$ of $\Phi$.  The forest expansion gives another representation of the balanced dichromatic polynomial.

\begin{thm}\label{T:tree}
The forest expansion is independent of $O$.  Indeed,
$$
 F_{\Phih,O}(\bu,y) = Q_\Phih(\bu,y-1,0) .
$$
\end{thm}

\begin{proof}
Let us expand.  In each sum, $S$ is restricted to balanced edge sets that satisfy the stated conditions.
\begin{align*}
Q_\Phih(\bu,v,0) 
 &= \sum_{S} v^{|S|-\rk(S)} \prod_{W\in\pi(S)} u_{h/S(W)} \\
 &= \sum_{F \text{ spanning forest}} \sum_{\substack{S\supseteq F\\ F(S)=F}}  v^{|S\setminus F|} \prod_{W\in\pi(S)} u_{h/S(W)} \\
 &= \sum_{F \text{ spanning forest}} \sum_{F\subseteq S \subseteq F\cup\EA(F)}  v^{|S\setminus F|} \prod_{W\in\pi(F)} u_{h/F(W)} 
\intertext{by Lemma \ref{L:basismatroid}, because $\pi(S)=\pi(F)$, and because $\EA(F)\subseteq\clos(F)$ so that $h/S(W) = h/{\clos(F)}(W) = h/F(W)$,}
 &= \sum_{F \text{ spanning forest}} (v+1)^{|\EA(F)|} \prod_{W\in\pi(F)} u_{h/F(W)} .
\qedhere
\end{align*}
\end{proof}

For an ordinary graph $\G$ (without gains or weights or loose or half edges), where all $u_h=u$, Theorem \ref{T:tree} says that 
$$
Q_\G(u,y-1) = F_\G(u,y) := \sum_F y^{\epsilon(F)} u^{c(F)},
$$ 
where $Q_\G$ is Tutte's dichromatic polynomial of $\G$.  (We do not know a source for this formula.)  A similar formula holds for the balanced dichromatic polynomial $Q^\textb_\Phi$ of a gain graph without weights.

We hoped for a spanning-tree expansion analogous to Tutte's for graphs, but we could not find one.  The problem is that semibases do not span the matroid.  When the semimatroid of graph balance is a matroid, as when $\Phi$ is balanced, a semibasis is a basis; then for a basis $T$ and an independent set $F$, $T(F) = T$ if and only if $\II(T) \subseteq F \subseteq T$.  This property lets us replace the sum over spanning forests by a sum over spanning trees.  We did not find an analogous property of semibases.

\section{Coloring}\label{coloration}

A proper coloration is a way of assigning to each vertex an element of a color set, subject to exclusion rules governed by the edges.  The subject of \cite{SOA} was the problem of counting integral lattice points not contained in specified integral affinographic hyperplanes (see Section \ref{affino}).  We solved it by reinterpreting lattice points as proper colorations of a $\bbZ$-weighted integral gain graph.  In this section we develop a theory of proper colorations of all weighted gain graphs.  
We begin with list coloring, where the weight of a vertex is a finite list of possible colors (Section \ref{list}).  We then go on to the heart of our treatment of coloring: infinite lists with an additional constraint regarded as a multidimensional variable; especially, upper bounds that make the effective list finite (Sections \ref{filtered} and \ref{intlatticelist}); this generalizes ordinary graph coloring, in which the list is $\{1,2,3,\ldots\}$ restricted by a variable upper bound $k$.  
Finally (Sections \ref{affino} and \ref{affinomatrix}), we apply the general definition to multidimensional integral gain groups and weights, which have the geometrical meaning of counting integer lattice points that lie in a given rectangular parallelepiped but not in any of a family of integral affinographic subspaces (all of which will be explained).

A notion of proper coloring of gain graphs was developed in \cite[Section III.5]{BG} (and there called zero-free coloring).  There is a \emph{color set} $\fC$, which is any set upon which the gain group $\fG$ has a fixed-point-free right action (that is, the only element of $\fG$ that has any fixed points is the identity).  
A \emph{coloration} of $\Phi$ is any function $x: V \to \fC$.  The set of \emph{improper} edges of $x$ is
$$
I(x) := \{ e_{ij} : x_j=x_i\phi(e_{ij}) \}.
$$ 
The coloration is \emph{proper} if $I(x)=\eset$.  
A basic fact from \cite[Section III.5]{BG} is that an improper edge set is balanced and closed.  For completeness we give the easy proof here.  (We remind the reader of the definitions in Section \ref{gaingraphs}.)

\begin{lem}\label{L:improper}
The improper edge set $I(x)$ of a coloration $x$ is balanced and closed.
\end{lem}

\begin{proof}
First we prove balance.  Suppose a circle $e_{01}e_{12}\cdots e_{l-1,l} \subseteq I(x)$, where $v_0=v_l$.  Impropriety of the edges implies that $x_l = x_0 \phi(e_{01}e_{12}\cdots e_{l-1,l})$.  Thus $x_0=x_l$ is a fixed point of $\phi(e_{01}e_{12}\cdots e_{l-1,l})$.  By our overall hypothesis that the action is fixed-point free, the circle has gain $\1$.  Thus, $I(x)$ is balanced.

Suppose now that $e$ is an edge from $v$ to $w$ in the closure of $I(x)$.  Since $I(x)$ is balanced, its closure is balanced (see Equation \eqref{E:bal-closure} and the accompanying text).  Thus, there is a path $e_{12}\cdots e_{l-1,l}$ in $I(x)$ connecting the endpoints of $e$ (that is, $v_1=v$ and $v_l=w$) whose gain $\phi(e_{12}\cdots e_{l-1,l}) = \phi(e)$.  Since $x_l = x_{l-1} \phi(e_{l-1,l}) = \cdots = x_1 \phi(e_{12})\cdots \phi(e_{l-1,l}) = x_1 \phi(e_{12}\cdots e_{l-1,l}) = x_1 \phi(e)$, $e$ is improper.  Hence, $e \in I(x)$; that is, $I(x)$ is closed.
\end{proof}

In contrast to \cite{BG}, in this paper we have an infinite color set.  We use weights in various ways to limit the possible colorations to a finite set.  We have especially in mind two kinds of example.  
In the first, the group and the color set are both $\bbZ$ and the color lists are arbitrary subsets of $\bbZ$ that are bounded below and whose complements are bounded above; but there is a variable upper bound $m$ on the possible colors; thus the number of proper colorations is a function of $m$.  We call this \emph{open-ended list coloring}.  
In the second, the group and color set are both $\bbZ^d$, the color lists are dual order ideals in $\bbZ^d$, and there is a variable upper bound $\bm_i$ on the colors that can be used at vertex $v_i$.  (This problem has a nice geometrical interpretation.)  
We wish to cover both of these examples, as well as similar ones, in a way that exposes to view the essential features; therefore we generalize considerably.

\subsection{List coloring}\label{list}

A simple kind of list coloring is the basis of all our methods of coloring a weighted gain graph.
The idea is to let $\fW$ be any class of subsets of the color set $\fC$ that is closed under intersection and the $\fG$-action, with intersection as its semigroup operation; 
these subsets can be used as vertex color lists.  A \emph{list coloration} or simply \emph{coloration} is any function $x: V \to \fC$ such that all $x(v_i) \in h_i$.  It is proper as a list coloration if it is proper as a coloration of its gain graph.

In list coloring a contracted weight $h/B$ has the formula 
$$
h/B(W) = \bigcap_{v_i \in W}  h_i \eta_B(v_i).
$$
This is the list-coloring interpretation of the definition in Equation \eqref{E:wtcontract}.  
(Recall that if $v_i \in W$ and it happens that $W$ has a top vertex $t_i$, then $\eta_B(v_i) = \phi(B_{v_it_i})$.)

We need to switch colorations.  If $x$ is a coloration of $\Phih$ and $\eta$ is a switching function, we define $x^\eta$ by 
$$
x^\eta(v_i) := x_i \eta(v_i),
$$ 
the result of the gain-group action on $x_i$.  

\begin{prop}\label{P:finlist}
If in $\Phih$ not all vertex lists $h_i$ are finite, then the number of proper colorations is either zero or infinite.  If all lists are finite, then the number of proper colorations equals
$$
\sum_{B \in \Latb \Phi} \mu(\eset,B) \prod_{W \in \pi(B)} |h/B(W)| 
 = \sum_{B \subseteq E: \text{ balanced}} (-1)^{{|B|}} \prod_{W \in \pi(B)} |h/B(W)| ,
$$
where $\mu$ is the M\"obius function of $\Latb\Phi$.   
\end{prop}

For the M\"obius function of a poset see, inter alia, \cite{FCT, EC1}; note that $\mu(\eset,B) = 0$ if the empty set is not closed, that is, if $\Phi$ has a balanced loop or a loose edge.  
The two sums are equal because $\mu(\eset,B) = \sum \{ (-1)^{|B'|} : \clos(B') = B \}$ if $B$ is balanced and closed, and $\pi(B')=\pi(B)$.  
For $\Latb\Phi$ see the end of Section \ref{gaingraphs}.

\begin{proof}
First let us suppose $h_n$ is infinite.  If there is any proper coloration $x$, then there are at most $n-1$ values that $x_1,\ldots,x_{n-1}$ prevent $x_n$ from taking on; but there is an infinite number of permitted possible choices of $x_n\in h_n$.

Now we assume all lists are finite.  To prove the first part of the formula we use M\"obius inversion over $\Latb\Phi$ as in \cite[p.\ 362]{FCT} or \cite[Theorem 2.4]{SGC}.  (The second part has a similar proof by inversion over the class of balanced edge sets.)  
Throughout the proof $B$ denotes an element of $\Latb\Phi$.  Consider all colorations of $\Phih$, proper or not; let $f(B)$ be the number of colorations $x$ such that $I(x)=B$ and let $g(B)$ be the number of colorations such that $I(x) \supseteq B$.  By Lemma \ref{L:improper} each coloration is counted in one $f(B)$, so for a closed, balanced set $A$, 
$$
g(A) = \sum_{B \supseteq A} f(B),
$$
from which by M\"obius inversion 
$$
f(A) = \sum_{B \supseteq A} \mu(\eset,B) g(B) .
$$
Setting $A = \eset$, the total number of proper colorations equals
$$
\sum_{B} \mu(\eset,B) g(B) .
$$

We show by a bijection that $g(B)$ is the number of all colorations of $\Phih/B$, which clearly equals $\prod_{W \in \pi(B)} |h/B(W)|$.  Let $\eta_B$ be the top switching function for $B$.  It is easy to see that, because $B$ is balanced, switching a coloration $x$ of $\Phih$ that is counted by $g(B)$ gives a coloration of $\Phih^{\eta_B}$ that is constant on components of $B$, and conversely.  Therefore, if $W \in \pi(B)$, $y_W$, defined as the common value of $x_i^{\eta_B}$ for every $v_i \in W$, belongs to $h_i^{\eta_B}$ for every $v_i \in W$.   When we contract $\Phih$ by $B$, $y_W \in \bigcap_{v_i \in W} h_i^{\eta_B} = h/B(W)$, so we get a well-defined coloration $y$ of $\Phih/B$.

Conversely, for any $B\in \Latb\Phi$, a coloration $y$ of $\Phih/B$ pulls back to a coloration of $\Phih^{\eta_B}$ by $x_i = y_W$ where $v_i\in W \in \pi(B)$.  Then switching back to $\Phih$ we have a coloration $x^{\eta_B\inv}$ of $\Phih$ whose improper edge set contains $B$.
Since it is clear that these correspondences are inverse to each other, the bijection is proved.
\end{proof}

The last part of the proof can be strengthened to yield a formula for proper colorations of contractions.

\begin{prop}\label{P:contractedproper}
If all vertex lists $h_i$ are finite and $B$ is a balanced edge set, then the number of colorations of $\Phih$ whose improper edge set equals $B$ equals the number of proper colorations of $\Phih/B$.
\end{prop}

The proof is a simple modification of the evaluation of $g(B)$ in the previous proof, and is also a simple generalization of the evaluation of $f(B)$ in the proof of \cite[Theorem 3.3]{SOA}.  
(In \cite{SOA} we accidentally wrote $f(B)$ when we meant $g(B)$; but that led us to write a proof of Proposition \ref{P:contractedproper} in the special situation of \cite{SOA}.  We thank Seth Chaiken for pointing out the error in \cite{SOA}.)

Let $\cPf(\fC)$ be the class of finite subsets of $\fC$, and let us call a weighted gain graph with weights in $\cPf(\fC)$ \emph{finitely list weighted}.  Then the total dichromatic polynomial has a variable $u_h$ for each finite subset $h \subseteq \fC$.

\begin{thm}\label{T:finlist}
If $\Phih$ is finitely list weighted, then the number of proper colorations equals $(-1)^n Q_{\Phih}(\bu,-1,0)$ evaluated at $u_{h} = -|h|$.
\end{thm}

Note that we evaluate the balanced dichromatic polynomial, not the total polynomial.  We do not have a similar interpretation of an evaluation of the total dichromatic polynomial.

\begin{proof}
The proof is by comparing the second formula of Proposition \ref{P:finlist} to the definition of $Q_{\Phih}$.
\end{proof}

Call a \emph{signed Tutte invariant} any function that satisfies (Tii--iv) and the modified form of (Ti),
\begin{enumerate}
 \item[(Ti${}^-$)] (Subtractivity)  For every link $e$, 
$$
 f(\Phi,h) = f(\Phi\setminus e,h) - f(\Phi/e,h/e) .
$$
\end{enumerate}
It is clear that $f$ is a signed Tutte invariant if and only if $(-1)^{|V|} f$ is a Tutte invariant.

\begin{cor}\label{C:finlistdc}
Given a gain group $\fG$ and a color set $\fC$, the number of proper colorations is a signed Tutte invariant of finitely list-weighted gain graphs with gains in $\fG$.
\end{cor}

\begin{proof}
The family of finitely list weighted $\fG$-gain graphs is closed under deletion and contraction because finiteness of lists is preserved by those operations.  Apply Theorem \ref{T:finlist}.
\end{proof}

\myexam{Finite lists}\label{X:finite}
If $\fC$ is partially ordered, take $\fW$ to consist of all finite order ideals, or all finite intervals.  The special case $\fC = \bbZ_{\geq0}^d$ is a main example (see Section \ref{affino}).
\end{exam}

\subsection{Filtered lists}\label{filtered}

In examples the color lists for the vertices are not always finite.  
A general picture is that for each vertex there is a fixed list $h_i$, which is some subset of the color set $\fC$, and there is also a variable set $M_i\subseteq\fC$ that acts as a filter of colors: a color must lie not only in its vertex list but also in $M_i$.  Thus we have a function of filters, $\chi_\Phih(M)$ defined for $M := (M_1,\ldots,M_n) \in \cP(\fC)^n$, whose value is the number of proper colorations of $\Phih$ using only colors in $M_i$ at vertex $v_i$.  This is the \emph{list chromatic function} of $\Phih$.  
An example, of course, is the quantity of Proposition \ref{P:finlist}, which equals $\chi_\Phih(\fC^n)$ (finite lists with no filtering).  
That proposition has the following extension.  We define switching of a color filter $M$ and its contraction $M/B(W) $ just as for weights, so that 
$$
M/B(W) := \bigcap_{v_i \in W}  M_i \eta_B(v_i).
$$

\begin{prop}\label{L:genlist}
If every intersection $M_i \cap h_i$ is finite, then 
$$
\chi_\Phih(M) = \sum_{B \in \Latb\Phi} \mu(\eset,B) \prod_{W \in \pi(B)}  | h/B(W) \cap M/B(W) | .
$$
\end{prop}

\begin{proof}
In Proposition \ref{P:finlist} replace $h_i$ by $h_i \cap M_i$.
\end{proof}

Since color filters contract like weights, we can form a \emph{doubly weighted} gain graph by taking new weights $(h_i,M_i)$, provided we define a new semigroup $\fM_0$ from which the double weights are drawn.  To that end, let 
$$
\fM_0 := \{ (h',M') : h' \in \fW,\ M' \subseteq \fC, \text{ and } M' \cap h' \text{ is finite} \} .
$$
The semigroup operation is componentwise intersection, i.e., $(h',M') \cap (h'',M'') := (h' \cap h'', M' \cap M'')$, and the action of $\fG$ on $\fM_0$ is componentwise.  
Given this weight semigroup, there is a doubly weighted total dichromatic polynomial, which we write $Q_{\Phih,M}(\bu,v,z)$, with $M = (M_1,\ldots,M_n)$, to emphasize the different roles of $h$ and $M$.  The variables are now $u_{h',M'}$ for each pair $(h',M') \in \fM_0$, and of course $v$ and $z$, and the formula for the doubly weighted polynomial is
$$
Q_{\Phih,M}(\bu,v,z) = 
 \sum_{S\subseteq E} v^{|S|-n+b(S)} z^{c(S)-b(S)} 
 \prod_{W\in\pib(S)} u_{h/S(W),M/S(W)} .
$$
where, as usual, $h/S(W) := h^{\eta_S}(w)$ and $M/S(W) := M^{\eta_S}(w)$ for $w \in W$.  It is easy to see that every $(h/S(W),M/S(W)) \in \fM_0$ if every $(h_i,M_i) \in \fM_0$.

\begin{thm}\label{T:genlist}
If $\Phih$ and $M_1,\ldots,M_n \subseteq \fC$ are such that the filtered list $M_i \cap h_i$ is finite for each vertex $v_i$, then the list chromatic function $\chi_\Phih(M_1,\ldots,M_n)$ is obtained from $(-1)^n Q_{\Phih,M}(\bu,-1,0)$ by setting $u_{h',M'} = -|M'\cap h'|$ for each $h' \in \fW$ and $M' \subseteq \fC$.
\end{thm}

\begin{proof}
Like that of Theorem \ref{T:finlist}, but from Proposition \ref{L:genlist}.  
\end{proof}

Fix the gain group $\fG$, color set $\fC$, and weight subsemigroup $\fW \subseteq (\cP(\fC),\cap)$.  Consider any weighted gain graph $\Phih$ with gains in $\fG$ and weights (functioning as vertex color lists) in $\fW$.  Allow any collection of color filters $(M_1,\ldots,M_n)$ for which all $h_i \cap M_i$ are finite; we call $(\Phih,M)$ a \emph{finitely filtered, list-weighted gain graph}.
This gives us a list chromatic function $\chi_\Phih(M_1,\ldots,M_n)$ that is always a well defined nonnegative integer.  Thinking of $(\Phih,M)$ as a (doubly) weighted gain graph, we have deletions and contractions and we can ask about Tutte invariance.

\begin{cor}\label{C:genlistdc}
The list chromatic function $\chi_\Phih(M_1,\ldots,M_n)$ is a signed Tutte invariant of finitely filtered, list-weighted gain graphs.
\end{cor}

\begin{proof}
As we noted, the class of list-weighted gain graphs with suitable arguments is closed under deletion and contraction.  
Now, apply Theorem \ref{T:genlist}.
\end{proof}

\myexam{Locally finite join semilattice}\label{X:lfposet}
Suppose $\fC$ is a locally finite join semilattice.  We may take $\fW$ to be the set of principal dual order ideals 
$$
\langle z \rangle^* := \{ y \in \fC : y \geq z \}
$$ 
and let the filters $M_i$ range over principal ideals
$$
\langle y \rangle := \{ x \in \fC : x \leq y \}.  
$$
The join operation makes $\fW$ an intersection semigroup, as the intersection $\langle z \rangle^* \cap \langle z' \rangle^*$ is the principal dual ideal $\langle z \vee z' \rangle^*$; and $\fW$ is clearly $\fG$-invariant.  The intersection $M_i \cap h_i$ will always be finite so the preceding proposition and theorem apply.  

We may weaken the assumptions.  Let $\fW$ be the class of all subsets of $\fC$ that have a lower bound; that is, subsets of principal dual ideals.  And let $M_i$ be any subset with an upper bound; that is, a subset of a principal ideal.  Then $M_i \cap h_i$ is necessarily finite so the preceding results hold good.  
That is, we admit as filters all upper-bounded subsets of $\fC$.

An example of this kind is that in which $\fC = \bbZ^d$; it is the topic of the next subsection.
\end{exam}

\subsection{Open-ended list coloring in an integer lattice}\label{intlatticelist}

To introduce the main applications we turn once again to the original example from \cite{SOA} but with a slight change in viewpoint.

\myexam{Open-ended interval coloring}\label{X:intervalcol}
The gain group is $\bbZ$ and the weights can be treated as upper intervals, $(h_i,\infty)$ at vertex $v_i$.  A \emph{proper $m$-coloration} of $\Phih$ is a function $x: V \to \bbZ$ such that each $x_i \in (h_i,\infty)$ and all $x_i \leq m$.
One can think of this as a list coloration in which the list for each vertex is an interval that grows with $m$.  
The \emph{integral chromatic function} $\chi^\bbZ_\Phih(m)$ of \cite{SOA} counts proper $m$-colorations.
This function is obtained from $Q(\bu,-1,0)$ by substituting $u_k=-\max(m-k,0)$.  Hence it is a signed Tutte invariant, as we showed in \cite{SOA} directly from its counting definition.  We showed in \cite{SOA} that it is eventually a monic polynomial of degree $n = |V|$, and that it is a sum of simple terms that appear successively as $m$ increases.
\end{exam}

We generalize this example in several ways: to higher-dimensional coloring, to upper bounds that depend on the vertex, and to arbitrary vertex lists.  

For higher-dimensional coloring the color set $\fC$ is the $d$-dimensional integer lattice $\bbZ^d$ with the componentwise partial ordering and the gain group $\fG$ is the additive group $\bbZ^d$ acting on $\fC$ by translation.  A coloration is any $\bx : V \to \bbZ^d$.  
The weight semigroup $\fW$ is either of the classes 
\begin{align*}
\fW_1 &:= \{ \H' \subset \bbZ^d : \H' \text{ is bounded below}\} 
\intertext{
---that is, $\H'$ is contained in a principal dual ideal 
$$\langle \ba \rangle^* = \bigtimes_{k=1}^{d} [a_k,\infty)$$ 
for some $\ba \in \bbZ^d$---and 
}
\fW_2 &:= \{ \H' \in \fW_1 : \text{ for some } \ba \in \bbZ^d \text{ and finite } A \subset \langle\ba\rangle^*, \ \H' = \langle \ba \rangle^* \setm A \} .
\end{align*}
Both classes, $\fW_1$ and $\fW_2$, are closed under intersection of pairs and under translation.  
For $\H' \in \fW_1$ we define $\bh'$ as the meet of the members of $\H'$---it is the largest possible $\ba$.  
When $\H'\in\fW_2$ and $d > 1$ this is the only possible $\ba$, but in dimension $d=1$ it is the smallest element of $\H'$.  

For $\H_i \in \fW_1$ define the \emph{dual-ideal complement} 
$$\bar\H_i := \langle \bh_i \rangle^* \setminus \H_i,$$ 
and let 
\begin{equation}
\hat\bh_i := \bigvee \bar\H_i,
\label{E:hath}
\end{equation}
except that $\hat\bh_i = \bh_i - (1,\ldots,1)$ if $\bar\H_i = \eset$.  Clearly, $\hat\bh_i$ is defined in $\bbZ^d$ if and only if $\bar\H_i$ is finite, that is, $\H_i \in \fW_2$.  

The color filters $M_i$ are principal order ideals 
$$
M_i = \langle \bm_i \rangle = \bigtimes_{k=1}^{d} (-\infty,m_{ik}], \text{ where } \bm_i = (m_{i1},\ldots,m_{id}) \in \bbZ^d,
$$ 
so the list chromatic function has domain $\bbZ^d$ and is defined by
\begin{align*}
\chi_\Phih(\bm) := \ &\text{the number of $\bx: V \to \bbZ^d$ such that each $\bx_i \in \H_i$},\ \bx \leq \bm, \\
&\text{and $\bx_j \neq \bx_i + \phi(e_{ij})$ for each edge $e_{ij}$} ,
\end{align*}
where $\bm := (\bm_1,\ldots,\bm_n)$.  
This number is a function of one variable $m_{ik}$ for each $i=1,\ldots,n$ and $k=1,\ldots,d$ and is finite for each $\bm$.  
Our general theory shows that $\chi_\Phih(\bm)$ is an evaluation of the total dichromatic polynomial; and now we can generalize Example \ref{X:intervalcol}.  In the present case a switching function is $\eta: V \to \bbZ^d$ and the contraction formula for weights takes the form 
$$
\H/B(W) = \bigcap_{v_i \in W} \big(\H_i+\eta_B(v_i)\big), \quad M_B(W) = \bigcap_{v_i \in W} \big(\langle \bm_i \rangle+\eta_B(v_i)\big)
$$
if $B$ is a balanced edge set and $W \in \pi(B)$, where $\H_i + \ba$ denotes the translation of $\H_i$ by $\ba$.  
In the total dichromatic polynomial there is one variable $u_{\H',\bm'}$ for each $\H' \in \fW_1$ and $\bm' \in (\bbZ^d)^n$.  (We simplify the notation $u_{\H',\langle\bm'\rangle}$ to  $u_{\H',\bm'}$.)

\begin{thm}\label{T:listQ}
With all $\H_i \in \fW_1$, the list chromatic function $\chi_\Phih(\bm)$ is obtained from $(-1)^n Q_{\Phih,M}(\bu,-1,0)$ by setting $u_{\H',\bm'} = -\big| \H' \cap \langle\bm'\rangle \big|$ for each $\H' \in \fW_1$ and $\bm' \in \bbZ^d$.
\end{thm}

\begin{proof}
A corollary of Theorem \ref{T:genlist}, since the intersections $\H_i \cap M_i$ are finite.
\end{proof}

\begin{cor}\label{C:listdcQ}
The list chromatic function is a signed Tutte invariant of  gain graphs with gain group $\bbZ^d$ and weight lists belonging to $\fW_1$.
\end{cor}

\begin{proof}
A special case of Corollary \ref{C:genlistdc}.
\end{proof}

For vertices $v_i$ and $v_j$, let 
\begin{equation}
\alpha_{ji}(\Phi) := \bigvee_{P_{ji}} \phi(P_{ji}) ,
\label{E:alphaij}
\end{equation}
where $P_{ji}$ ranges over all paths in $\Phi$ from $v_j$ to $v_i$.  

A function $p(y_1,\ldots,y_r)$ defined on $\bbR^r$ is a \emph{piecewise polynomial} if $\bbR^r$ is a union of a finite number of closed, full-dimensional sets $D_\sigma$, on each of which $p$ is a polynomial $p_\sigma(y_1,\ldots,y_r)$.  
By saying $p$ has leading term $y_1\cdots y_r$, we mean that at least one $p_\sigma$ has that leading term and every other $p_\sigma$ has degree no higher than 1 in any variable.  
This definition is chosen to suit the following result.  The numbers $h_{ik}$ are the components of the vectors $\bh_i$.  This theorem looks complicated; the reader may find Example \ref{X:order2}, after the proof, helpful.

\begin{thm}\label{T:orthotope}
Assume $\Phih$ has no balanced loops or loose edges.  
Suppose all $\H_i \in \fW_2$.  
Define
$$
q_k(B,W) := \min_{v_i \in W} \big(m_{ik}+\eta_B(v_i)_k\big) - \max_{v_i \in W} \big(h_{ik}+\eta_B(v_i)_k\big)  + 1
$$
for $B \in \Latb\Phi$, $W \in \pi(B)$, and $1 \leq k \leq d$, and
\begin{align*}
p(\bm) := \sum_{B \in \Latb\Phi} \mu(\eset,B) \prod_{W \in \pi(B)}  \Big(  \prod_{k=1}^d  q_k(B,W)  -  \big| \bigcup_{v_i \in W} \big( \bar\H_i + \eta_B(v_i) \big)  \big|  \Big) ,
\end{align*}
which is a piecewise polynomial function of the $nd$ variables $m_{ik}$, having $\prod_{i=1}^n \prod_{k=1}^d m_{ik}$ as its leading term.  The list chromatic function $\chi_\Phih(\bm)$ is a piecewise polynomial that equals $p(\bm)$ for all $\bm \in (\bbZ^d)^n$ such that $\bm \geq \bar{\hat\bm}\Phih$, where 
\begin{equation}
\bar{\hat\bm}_i\Phih := \bigvee_{j=1}^n \big( \hat\bh_j + \alpha_{ji}(\Phi) \big) \in \bbZ^d .
\label{E:barhatm}
\end{equation}
\end{thm}

For the ordinary graph version of this theorem, with weights but without gains, see Corollary \ref{C:graph}.

\begin{proof}
In this theorem we have $r=nd$, the variables are $m_{11},m_{12},\ldots,m_{1d},m_{21},\ldots,m_{nd}$, and the domains $D_\sigma$ are the sets on which each of the $d$ sets (one for each fixed $k \leq d$) of shifted variables $m_k(B,W) := \min_{v_i \in W}  \big(m_{ik}+\eta_B(v_i)\big)$ (one variable for each $B$, each $W \in \pi(B)$, and each $k\leq d$) assumes a particular weakly increasing order, since the orderings of these variables determine exactly which polynomial $p(\bm)$ is.  Thus, $p(\bm)$ is definitely a piecewise polynomial.  One task is to prove that it equals $\chi_\Phih(\bm)$ for $\bm \geq \bar{\hat\bm}\Phih$.  The other, which we do first, is to prove that $\chi_\Phih(\bm)$ itself is a piecewise polynomial.

We apply the formula of Proposition \ref{L:genlist} in the form
\begin{equation}\label{E:ortholist}
\chi_\Phih(\bm) = \sum_{B \in \Latb\Phi} \mu(\eset,B) \prod_{W \in \pi(B)} \big|\H/B(W) \cap M/B(W)\big| ,
\end{equation}
which shows that the theorem is true when the range of $\bm$ is such that $|\H/B(W) \cap M/B(W)| = q_k(B,W)$. 
The right-hand factor will be made more explicit in Equation \eqref{E:factor}.

First we partially determine the values of $\bm$ for which the term of a closed, balanced set $B$ is nonzero.  Define $\bar\bm_i\Phih := \bigvee_{j=1}^n \big( \bh_j + \alpha_{ji}(\Phi) \big)$.  Note that 
$$\bar\bm\Phih = \bigvee_{B\in\Latb\Phi} \bar\bm(\Phih|B),$$ 
since every path $P_{ji}$ is contained in some balanced set $B$.  Restrict $\bar\bm_i(\Phih|B)$ to $W$ as $\bar\bm(\Phih|B)|_W$.

\begin{lem}\label{L:orthozero}
In the sum of Equation \eqref{E:ortholist}, the term of $B$ is zero unless $\bm \geq \bar\bm(\Phih|B)$.
\end{lem}

\begin{proof}
Rewrite the expression for the factor of $W$ in the term of $B$, with $[\bx,\by]$ denoting an interval in the lattice $\bbZ^d$ and $\bm_i+\eta_B(v_i)$ meaning that $\bm_i$ is translated by $\eta_B(v_i)$:
\begin{equation}\label{E:factor}
\begin{aligned}
\H/B(W) \cap M/B(W)
 &=  \bigcap_{v_i \in W} (\H_i+\eta_B(v_i))  \cap  \bigcap_{v_i \in W} \langle \bm_i+\eta_B(v_i) \rangle \\
 &=  \bigcap_{v_i \in W} \big( [\langle \bh_i \rangle^* \setm \bar\H_i]+\eta_B(v_i) \big)  \cap  \bigcap_{v_i \in W} \big( \langle \bm_i \rangle+\eta_B(v_i)\big) \\
 &=  \bigcap_{v_i \in W} \Big( \big( [ \bh_i, \bm_i ] \setm \bar\H_i \big) +\eta_B(v_i) \Big) \\
 &=  \bigcap_{v_i \in W} \Big( [ \bh_i, \bm_i ] +\eta_B(v_i) \Big) \setminus \bigcup_{v_i \in W} \big( \bar\H_i +\eta_B(v_i) \big)  \\
 &=  \bigcap_{v_i \in W}  [ \bh_i+\eta_B(v_i), \bm_i+\eta_B(v_i) ]  \setminus \bigcup_{v_i \in W} \big( \bar\H_i +\eta_B(v_i) \big)  \\
 &=   \big[ \bigvee_{v_i \in W} \big( \bh_i+\eta_B(v_i) \big), \bigwedge_{v_i \in W} \big( \bm_i+\eta_B(v_i) \big) \big]   \setm  \bigcup_{v_i \in W} \big( \bar\H_i + \eta_B(v_i) \big) 
\end{aligned}
\end{equation}
since $\H_i = \langle \bh_i \rangle^* \setm \bar\H_i$ and $M_i = \langle \bm_i \rangle$.  The interval in the last line has another expression:
\begin{align*}
&\big[ \bigvee_{v_i \in W} \big(\bh_i+\eta_B(v_i)\big), \bigwedge_{v_i \in W} \big(\bm_i+\eta_B(v_i)\big) \big] \\
&= \bigtimes_{k=1}^d \big[\max_{v_i\in W}\big(h_{ik}+\eta_B(v_i)_k\big), \min_{v_i\in W}\big(m_{ik}+\eta_B(v_i)_k\big)\big],
\end{align*}
which is an orthotope that plainly contains $\prod_{k=1}^d q_k(B,W)$ lattice points if it is not empty.

Let us see, therefore, what the conditions are under which the orthotope is nonempty.  We need $\bh_i+\eta_B(v_i) \leq \bm_j+\eta_B(v_j)$ for all $v_i,v_j \in W$.  That means $\bm_j \geq \bh_i + \eta_B(v_i) - \eta_B(v_j).$
Now recall from Equation \eqref{E:etatransfer} that $\eta_B(v_j) - \eta_B(v_i) = \phi(B_{ji})$.  Here $B_{ji}$ is any path in $B$ from $v_j$ to $v_i$.  
Thus, the condition on $\bm_j$ is that $\bm_j \geq \bh_i + \alpha_{ji}(B\ind W)$, which is satisfied for all $v_j \in W$ if and only if 
$\bm|_W \geq \bar\bm(\Phih|B)|_W.$
If this is not satisfied for all $W \in \pi(B)$, that is, if $\bm \not\geq \bar\bm(\Phih|B)$, then the orthotope is empty and the term of $B$ disappears.  
\end{proof}

If $\bm \geq \bar\bm(\Phih|B)$, the term of $B$ may still equal 0 if the orthotope is contained in $\bigcup_{v_i \in W} \big( \bar\H_i + \eta_B(v_i) \big)$; therefore, the converse of the lemma is not valid.

The lemma implies that terms of $\chi_\Phih(\bm)$ appear gradually as $\bm$ gets larger.  That generalizes the behavior of the corresponding function for $d=1$ (Example \ref{X:intervalcol}), proved in \cite{SOA}.  The condition for all terms to appear is that $\bm \geq \bigvee_B \bar\bm(\Phih|B) = \bar\bm\Phih.$

Next, we prove piecewise polynomiality of $\chi_\Phih(\bm)$ (with the specified leading term).  Indeed, each factor in each term is piecewise linear.  

Consider first $W = \{v_i\}$, which is a component when $B = \eset$.  (Then $\eta_B(v_i) = 0 \in \bbR^d$.)  The corresponding factor is 
\begin{equation}\label{E:vertexfactor}
\begin{aligned}
\big| \H(\{v_i\}) \cap M(\{v_i\}) \big| &= \big|  \big( [\bh_i, \bm_i] + \eta_B(v_i) \big)  \setm  \big( \bar\H_i + \eta_B(v_i) \big)  \big| 
 =  \big|  [\bh_i, \bm_i]  \setm  \bar\H_i  \big| \\
 &\overset?= \big|  [\bh_i, \bm_i]  \big|  -  \big| \big( \bar\H_i \cap \langle \bm_i \rangle \big)  \big| \\
 &=  \prod_{k=1}^d  (m_{ik} - h_{ik} + 1)^+  -  \big| \bar\H_i \cap \langle \bm_i \rangle \big| .
\end{aligned}
\end{equation}
When $\bm_i \geq \hat\bh_i$, the questionable equality $\overset?=$ is equality and moreover the right-hand side is a multilinear function with leading term $\prod_{i=1}^n \prod_{k=1}^d m_{ik}$, because then $\bar\H_i \subseteq \langle \bm_i \rangle$ so $| \bar\H_i \cap \langle \bm_i \rangle | = |\bar\H_i|$.
Furthermore, for any $\bm$, $\bar\H_i \cap \langle \bm_i \rangle$ is a finite subset $S \subseteq \bar\H_i$, and for each $S$, $ | [\bh_i, \bm_i]  \setm  \big( \bar\H_i \cap \langle \bm_i \rangle \big) | = | [\bh_i, \bm_i] |  -  | S | $, which is a multilinear function of $\bm$.  There are only finitely many subsets $S$, hence $(\bbZ^d)^n$ is the union of a finite number of domains, on each of which $\chi_\Phih(\bm)$ is multilinear.  It follows that the term of $B=\eset$ is piecewise polynomial on $(\bbR^d)^n$ and has the correct leading term.

Treating in the same manner each factor in Equation \eqref{E:factor}, rewritten as 
\begin{equation}\begin{aligned}\label{E:setfactor}
\H/B(W) \cap M/B(W)
 &=   \big[ \bigvee_{v_i \in W} \big(\bh_i+\eta_B(v_i)\big), \bigwedge_{v_i \in W} \big(\bm_i+\eta_B(v_i)\big) \big]  \setm  S 
\end{aligned}\end{equation}
where $S :=  \bigcup_{v_i \in W} \big( \bar\H_i + \eta_B(v_i) \big) \cap \bigcap_{v_i \in W} \big\langle \bm_i+\eta_B(v_i) \big\rangle ,$ 
we see that $| \H/B(W) \cap M/B(W) |$ is similarly piecewise linear since there are only finitely many possible sets $S$ and each one is finite.  When $S = \bigcup_{v_i \in W} \big( \bar\H_i + \eta_B(v_i) \big)$, i.e., when $\bm_i+\eta_B(v_i) \geq \bh_j+\eta_B(v_j)$ for all $v_i,v_j \in W$, the formula is 
\begin{equation*}\begin{aligned}
\big| \H/B(W) \cap M/B(W) \big|
 &=   \big| \big[ \bigvee_{v_i \in W} \big(\bh_i+\eta_B(v_i)\big), \bigwedge_{v_i \in W} \big(\bm_i+\eta_B(v_i)\big) \big] \big|  -  \big| S \big| \\
 &=  q_k(B,W) -  \big| \bigcup_{v_i \in W} \big( \bar\H_i + \eta_B(v_i) \big) \big| .
\end{aligned}\end{equation*}
Recall from Equation \eqref{E:etatransfer} that $\eta_B(v_j) - \eta_B(v_i) = \phi(B_{ji})$, where $B_{ji}$ is a path in $B$ from $v_j$ to $v_i$.  
We can choose $B$ to contain any desired path $P_{ji}$ from $v_j$ to $v_i$ in $\Phi$ by letting $B$ be a balanced set containing $P_{ji}$.  Thus, the condition on $\bm_j$ is that $\bm_j \geq \bh_j + \bigvee_{P_{ji}} \phi(P_{ji}) = \bh_j + \alpha_{ji}(\Phi)$, which is satisfied for all $v_j \in V$ if and only if $\bm \geq \bar\bm\Phih.$

We have proved that $\chi_\Phih(\bm)$ is a piecewise polynomial function.  The term of highest degree is that for which $\pi(B)$ has the most components, which is $n$ components when $B = \eset$.  The term of each balanced, closed set $B$ is a monic piecewise polynomial of total degree $dc(B)$.
The previous analysis shows that the natural sufficient condition for $\chi_\Phih(\bm)$ to equal $p(\bm)$ is that $S = \bar\H_i \cap \langle\bm_i\rangle$ in every factor of every term, i.e., $\bm_i \geq \bar{\hat\bm}_i\Phih$.
Thus, to make the chromatic function equal to $p(\bm)$ it suffices that $\bm \geq \bar{\hat\bm}\Phih$.  
\end{proof}

The proof suggests that the theorem's lower bound on $\bm$ is essential; for any other choice of $\bm$, $p(\bm)$ will not agree with $\chi_\Phih(\bm)$.  The main reason is that the constant term $\big| \bigcup_{v_i \in W} \big( \bar\H_i + \eta_B(v_i) \big)  \big|$ in the factor of $W$ in $p(\bm)$ assumes that $\bigcup_{v_i \in W} \big( \bar\H_i + \eta_B(v_i) \big)$ is contained in $\big\langle \bigwedge_{v_i \in W} \big(\bm_i+\eta_B(v_i)\big) \big\rangle$.   However, we have not tried to prove necessity of the lower bound, and there might be exceptions.

It is clear, though, why $\H_i$ has to be bounded below and its dual-ideal complement $\bar\H_i$ must be finite.  If some $\H_i$ has no lower bound, then $\chi_\Phih(m)$ will be infinite.  Even when each $\H_i$ is bounded below, if some $\bar\H_i$ is infinite it has no upper bound; then $\chi_\Phih(m)$ will not settle down to a finite system of polynomials for large values $\bm$.

\begin{exam}\label{X:order2}
We do a very small example, just large enough to show the phenomena that appear in computing $\chi_\Phih(\bm)$.  
This example has gain graph $\Phi$ with vertex set $V=\{v_1,v_2\}$ ($n=2$ vertices), so it represents hyperplanes in $\bbR^2$.  The gain group and color set are $\bbZ^2$ ($d=2$).  
The edge set is 
$$
E = \{e_1=(0,0)(v_1,v_2),\ e_2=(2,0)(v_1,v_2),\ e_3=(-1,2)(v_1,v_2)\}.
$$  
The variables are four, in two lattice points: $\bm=(\bm_1,\bm_2)$, where $\bm_1=(m_{11},m_{12})$ and $\bm_2=(m_{21},m_{22})$.  We choose 
$$\bh_1=(1,0),\ \bh_2=(0,3),\ \bar\H_1=\{\bc_1\},\ \bar\H_2=\{\bc_2\},$$
where $\bc_i \in \bbZ^2$ satisfies $\bc_i\geq\bh_i$.  
Thus, $\hat\bh_i=\bigvee\bar\H_i=\bc_i$ and $\bar{\hat\bm}_i = [\hat\bh_1+\alpha_{1i}] \vee [\hat\bh_2+\alpha_{2i}]$ for $i=1,2$.  
(For the definition of ${\hat\bh}_i$ see Equation \eqref{E:hath}.  For that of $\bar{\hat\bm}_i$ see \eqref{E:barhatm}.)  
We evaluate the $\alpha_{ji}$ from Equation \eqref{E:alphaij}:
\begin{equation*}
\begin{array}{ll}
\alpha_{11} = \bigvee \phi(P_{11}) = (0,0),	&\alpha_{12} = \phi(e_1) \vee \phi(e_2) \vee \phi(e_3) = (2,2), \\[6pt]
\alpha_{21} = \phi(e_1\inv) \vee \phi(e_2\inv) \vee \phi(e_3\inv) = (1,0),	&\alpha_{22} = \bigvee \phi(P_{22}) = (0,0).
\end{array}
\end{equation*}
Therefore, 
\begin{equation}
\begin{aligned}
\bar{\hat\bm}_1 
&= (\bc_1+(0,0)) \vee (\bc_2+(1,0)) = (\max(c_{11},c_{21}+1), \max(c_{12},c_{22})),
\\
\bar{\hat\bm}_2 
&= (\bc_1+(2,2)) \vee (\bc_2+(0,0)) = (\max(c_{11}+2,c_{21}), \max(c_{12}+2,c_{22})).
\end{aligned}
\label{E:Xorder2bhm}
\end{equation}

Recall the notations $x^+, \bx^+, x^-, \bx^-$ from the beginning of Section \ref{defs}.  
We employ the zeta function of $\bbZ^2$ defined by $\zeta(\bx,\by) := 1$ if $\bx\leq\by$, otherwise 0; and 
the set membership function $\epsilon(a,S) := 1$ if $a \in S$, otherwise 0.

To evaluate $\chi_\Phih(\bm)$ we use the combination of Equations \eqref{E:ortholist} and \eqref{E:factor}.  We first find the factor of each $W\in\pi(\eset)$.  For $W = \{v_1\}$ the factor is  
$| [\bh_1,\bm_1] \setm \{\bc_1\} |.$
For $W = \{v_2\}$ it is  
$| [\bh_2,\bm_2] \setm \{\bc_2\} |.$
Hence the term of $B=\eset$ is the product:
\begin{equation}\begin{aligned}\label{E:eset}
&\mu(\eset,\eset) \cdot \big| [\bh_1,\bm_1] \setm \{\bc_1\} \big| \cdot \big| [\bh_2,\bm_2] \setm \{\bc_2\} \big| \\
&= 1 \cdot \big( (m_{11}-h_{11}+1)^+(m_{12}-h_{12}+1)^+ - \epsilon(\bc_{1},[\bh_{1},\bm_{1}]) \big)^+  \\
&\qquad \cdot  \big( (m_{21}-h_{21}+1)^+(m_{22}-h_{22}+1)^+ - \epsilon(\bc_{2},[\bh_{2},\bm_{2}]) \big)^+ \\
&= \big( m_{11}^+(m_{12}+1)^+ - \epsilon(\bc_{1},[(1,0),\bm_{1}]) \big)^+ \\
&\qquad \cdot  \big( (m_{21}+1)^+(m_{22}-2)^+ - \epsilon(\bc_{2},[(0,3),\bm_{2}]) \big)^+ .
\end{aligned}\end{equation}
The term of $B=\{e_i\}$, consisting of any one edge $e_i=\bg(v_1,v_2)$, requires computing $\eta_B(v_i)$ by means of Lemma \ref{L:top}.  Thus, 
\begin{align*}
\eta_B(v_1)&=\phi(B_{v_1v_1})\vee\phi(B_{v_1v_2}) = (0,0)\vee \bg = \bg^+, \\
\eta_B(v_2)&=\phi(B_{v_2v_1})\vee\phi(B_{v_2v_2}) = (-\bg)\vee(0,0) = \bg^-.
\end{align*}
The term of this set $B$ is therefore 
\begin{align*}
&\mu(\eset,\{e_i\}) \left| \big[ (\bh_1+\bg^+)\vee(\bh_2+\bg^-), (\bm_1+\bg^+)\wedge(\bm_2+\bg^-) \big] \setm \{ \bc_1+\bg^+, \bc_2+\bg^- \} \right| \\
&= (-1) \left| \big[ ( \max(h_{11}+g_1^+, h_{21}+g_1^-), \max(h_{12}+g_2^+, h_{22}+g_2^-) ), \right.\\
&\qquad\qquad\quad ( \min(m_{11}+g_1^+,m_{21}+g_1^-), \min(m_{12}+g_2^+,m_{22}+g_2^-) ) \big] \\
&\qquad\qquad\left. \setm\ \{ \bc_1+\bg^+, \bc_2+\bg^- \} \right| \\
&= - \big\{ \big( \min(m_{11}+g_1^+,m_{21}+g_1^-) - \max(h_{11}+g_1^+,h_{21}+g_1^-) + 1 \big)^+ \\
&\qquad\quad\quad \cdot \big( \min(m_{12}+g_2^+,m_{22}+g_2^-) - \max(h_{12}+g_2^+,h_{22}+g_2^-) + 1 \big)^+ \\
&\qquad\quad  - \epsilon(\bc_1+\bg^+,[ (\bh_1+\bg^+)\vee(\bh_2+\bg^-), (\bm_1+\bg^+)\wedge(\bm_2+\bg^-) ]) \\
&\qquad\quad  - \epsilon(\bc_2+\bg^-,[ (\bh_1+\bg^+)\vee(\bh_2+\bg^-), (\bm_1+\bg^+)\wedge(\bm_2+\bg^-) ]) 
\big\}  \\
&= - \big\{ \big( \min(m_{11},m_{21}-g_1) - \max(h_{11},h_{21}-g_1) + 1 \big)^+ \\
&\qquad\quad\quad \cdot \big( \min(m_{12},m_{22}-g_2) - \max(h_{12},h_{22}-g_2) + 1 \big)^+ \\
&\qquad\quad  - \epsilon(\bc_1,[ \bh_1\vee(\bh_2-\bg), \bm_1\wedge(\bm_2-\bg) ]) \\
&\qquad\quad  - \epsilon(\bc_2,[ (\bh_1+\bg)\vee\bh_2, (\bm_1+\bg)\wedge\bm_2 ]) 
\big\}.
\end{align*}
For instance, if $\bg=(0,0)$ with our choices of the $\bh_i$, the term of $B = \{(0,0)(v_1,v_2)\}$ is 
\begin{equation}\begin{aligned}\label{E:00e12}
& - \big\{ \big( \min(m_{11},m_{21}) \big)^+  \cdot \big( \min(m_{12},m_{22}) - 2 \big)^+ \\
&\qquad\quad  - \epsilon(\bc_1,[(1,3), \bm_1\wedge\bm_2 ]) - \epsilon(\bc_2,[(1,3), \bm_1\wedge\bm_2 ]) 
\big\}.
\end{aligned}\end{equation}
If $\bg=(2,0)$, the term of $B = \{(2,0)(v_1,v_2)\}$ is 
\begin{equation}\begin{aligned}\label{E:20e12}
& - \big\{ \big( \min(m_{11},m_{21}-2) - \max(1,-2) + 1 \big)^+ \cdot \big( \min(m_{12},m_{22}) - \max(0,3) + 1 \big)^+ \\
&\qquad  - \epsilon(\bc_1,[ (1,0)\vee((0,3)-(2,0)), \bm_1\wedge(\bm_2-(2,0)) ]) \\
&\qquad  - \epsilon(\bc_2,[ ((1,0)+(2,0))\vee(0,3), (\bm_1+(2,0))\wedge\bm_2 ]) 
\big\} \\
&= - \big\{ \big( \min(m_{11},m_{21}-2) \big)^+ \cdot \big( \min(m_{12},m_{22}) - 2 \big)^+ \\
&\qquad\quad  - \epsilon(\bc_1,[ (1,3), \bm_1\wedge(\bm_2-(2,0)) ]) - \epsilon(\bc_2,[ (3,3), (\bm_1+(2,0))\wedge\bm_2 ]) 
\big\} .
\end{aligned}\end{equation}
If $\bg=(-1,2)$, the term of $B = \{(-1,2)(v_1,v_2)\}$ is 
\begin{equation}\begin{aligned}\label{E:-12e12}
&- \big\{ \big( \min(m_{11},m_{21}+1) \big)^+ \cdot \big( \min(m_{12},m_{22}-2) \big)^+ \\
&\qquad  - \epsilon(\bc_1,[ (1,1), \bm_1\wedge(\bm_2-(-1,2)) ]) \\
&\qquad  - \epsilon(\bc_2,[ (0,3), (\bm_1+(-1,2))\wedge\bm_2 ]) 
\big\}.
\end{aligned}\end{equation}
The degree of each term in each variable $m_{ij}$ is 0 or 1, depending on how the shifted value of $m_{ij}$ is related to the shifted values of other variables.  

Now we apply these calculations to evaluate the list chromatic function:  
\begin{equation*}\begin{aligned}\label{E:order2}
\chi_\Phih(\bm) &= \eqref{E:eset} + \eqref{E:00e12} + \eqref{E:20e12} + \eqref{E:-12e12} .
\end{aligned}\end{equation*}
As a function of $\bm$, this consists of different polynomials in different domains, depending in part on the relative ordering of the variables $m_{ij}$ and shifted variables $m_{11}+2$, etc.

For \emph{large values} of all the $m_{ij}$, $\chi_\Phih(\bm)$ has several domains on which it is different polynomials, all of which have leading term $m_{11}m_{12}m_{21}m_{22}$.  We are assuming specifically that all the minima of translated $m_{ij}$ are larger than all the maxima of translated $h_{ij}$ in the preceding formulas.  In terms of $\bar{\hat\bm}\Phih$, we are assuming 
$\bm_1 \geq \bar{\hat\bm}_1$ and $\bm_2 \geq \bar{\hat\bm}_2$, lower bounds whose values were computed in Equation \eqref{E:Xorder2bhm}.  
Every term of $\chi_\Phih(\bm)$ will be linear in each of the $m_{ij}$.  The first term is
\begin{equation*}\begin{aligned}
\eqref{E:eset}
&= \big( (m_{11}-h_{11}+1)(m_{12}-h_{12}+1) - \epsilon(\bc_{1},\langle \bh_1 \rangle^*) \big)  \\
&\qquad \cdot  \big( (m_{21}-h_{21}+1)(m_{22}-h_{22}+1) - \epsilon(\bc_{2},\langle \bh_2 \rangle^*) \big) \\
&= \big( m_{11}(m_{12}+1) - \zeta((1,0),\bc_1) \big) \cdot  \big( (m_{21}+1)(m_{22}-2) - \zeta((0,3),\bc_2) \big) 
\end{aligned}\end{equation*}
(note that $\zeta(\by,\bx)=\epsilon(\bx,\langle\by\rangle^*)$), which shows (since the $h_{ij}$ and the $\epsilon$'s are independent of $\bm$) that, no matter how the shifted and unshifted variables are related, the term $m_{11}m_{12}m_{21}m_{22}$ exists; it will be the leading term since the remaining terms of $\chi_\Phih(\bm)$ have total degree 2.  The remaining terms depend on how the variables are ordered.  For instance, 
\begin{equation*}\begin{aligned}
\eqref{E:00e12} 
&=  \big( \min(m_{11},m_{21}) - \max(h_{11},h_{21}) + 1 \big) 
\big( \min(m_{12},m_{22}) - \max(h_{12},h_{22}) + 1 \big) \\
&\quad  - \epsilon(\bc_1,\langle \bh_1\vee\bh_2 \rangle^*) 
  - \epsilon(\bc_2,\langle \bh_1\vee\bh_2 \rangle^*) \\
&= \left. \begin{cases}
  \begin{aligned}
  & m_{11}  \big( m_{12} - 2 \big) 
  \end{aligned} 
&\text{ if } m_{11}\leq m_{21},\ m_{12}\leq m_{22}, \\
  \begin{aligned}
  & m_{21}  \big( m_{12} - 2 \big) 
  \end{aligned}  
&\text{ if } m_{11}\geq m_{21},\ m_{12}\leq m_{22}, \\
  \begin{aligned}
  & m_{11}  \big( m_{22} - 2 \big)
  \end{aligned}  
&\text{ if } m_{11}\leq m_{21},\ m_{12}\geq m_{22}, \\
  \begin{aligned}
  & m_{21}  \big( m_{22} - 2 \big) 
  \end{aligned}  
&\text{ if } m_{11}\geq m_{21},\ m_{12}\geq m_{22}
\end{cases} \right\} \\
&\quad - \zeta((1,3),\bc_1) - \zeta((1,3),\bc_2) .
\end{aligned}\end{equation*}
In each of these four domains governed by the relative values of the $m_{ij}$, $\chi_\Phih(\bm)$ is a different polynomial that is linear in all four variables and has total degree 4.  That concludes the example.
\end{exam}

Theorem \ref{T:orthotope} simplifies when all $\bm_i$ equal a common value $\bm' \in \bbZ^d$.  Define
$$
\alpha_j(\Phi) := \bigvee_i \alpha_{ji}(\Phi),
$$ 
the least upper bound of the gains of all paths that begin at $v_j$.  

\begin{cor}\label{C:onebound}
For a $\bbZ^d$-gain graph $\Phi$ with no balanced loops or loose edges, 
suppose all weights $\H_i \in \fW_2$.  
Define
$$
\bar q_k(B,W) := m'_k - \max_{v_i, v_j \in W} \big( h_{ik}+\eta_B(v_i)_k-\eta_B(v_j)_k \big)  + 1
$$
for $B \in \Latb\Phi$, $W \in \pi(B)$, and $\bm' \in \bbZ^d$.  
For large enough $\bm' \in\bbZ^d$, the list chromatic function $\chi_\Phih(\bm',\ldots,\bm')$ equals 
\begin{align*}
\bar p(\bm') := \sum_{B \in \Latb\Phi} \mu(\eset,B) \prod_{W \in \pi(B)}  &\bigg(  \prod_{k=1}^d  \bar q_k(B,W)  -  \Big| \bigcup_{v_i \in W} \big( \bar\H_i + \eta_B(v_i) \big)  \Big|  \bigg) ,
\end{align*}
which is a polynomial function of the $d$ variables $m'_k$ having degree at most $n$ in each variable and leading term $\prod_{k=1}^d (m'_k)^n$.  The equation $\chi_\Phih(\bm',\ldots,\bm') = \bar p(\bm')$ holds true for all 
$$
\bm' \geq \bigvee_{j=1}^n \big( \hat\bh_j + \alpha_j(\Phi) \big) .
$$
\end{cor}

In the one-dimensional case, where the gain group and color set are $\bbZ$, then 
$$
\fW_2 = \{ \H' \subset \bbZ : \H' \text{ is bounded below and } \bbZ \setm \H' \text{ is bounded above}\}
$$
and $\bm = (m_1,\ldots,m_n) \in \bbZ^n$.  
In this case $\bh_i = \min \H_i$ and $\hat\bh_i = \max(\bbZ \setm \H_i)$; and $\alpha_{ji}(\Phi)$ is the largest gain of a path in $\Phi$ from $v_j$ to $v_i$.

\begin{cor}\label{C:orthotope1}
For an integral gain graph without balanced loops or loose edges, and with all weights $\H_i \in \fW_2$, $\chi_\Phih(m_1,\ldots,m_n)$ is a monic polynomial in the $n$ variables $m_i$ for large enough $m_i$'s, linear in each variable and with leading term $m_1 \cdots m_n$.  Polynomiality holds when all 
$$
m_i \geq m_{0i} := \max_{j=1,\ldots,n} \big[ \max(\bbZ \setm \H_j) + \alpha_{ji}(\Phi) \big] .
$$
\end{cor}

Now we begin to justify the claim that a total dichromatic polynomial connected with an interesting combinatorial problem has an uncountable number of variables.  
(We assume the reader finds the list chromatic function with gains in $\bbZ$ or $\bbZ^d$ interesting, or this argument fails!  The geometrization in the next subsection may add to the interest.)  
The number of variables $u_{\H',\bm'}$ when the weight semigroup is $\fW_1$ (and $\bm' \in \bbZ^d$) is $|\fM_0| = |\fW_1| \cdot |\bbZ^d|$.  As $\fW_1$ contains every subset of the natural numbers, its cardinality is that of the continuum.

On the other hand, $|\fW_2| = \aleph_0$, for $\fW_2$ is a countable union of countable sets.  
We see this by describing $\H'\in\fW_2$ as an ordered pair $(\bh', \bar\H')$.  There is a countable number of pairs $(\ba,X)$ of this type, for each $\ba$, and the number of integer vectors $\ba$ is countable.

\subsection{Arrangements of affinographic hyperplanes}\label{affino}

Finally, we prove the geometrical theorems stated or mentioned in the introduction.  

We state the exact version of Theorem \ref{T:geomorthotope}.  As usual, $\Phi$ is the gain graph corresponding to $\cA$; $t_k$ is a top vertex of the component $B_k$ of $B$, whose vertex set is $W_k$; and $\mu$ is the M\"obius function of the semilattice $\Latb\Phi$.  By $(\Phi,0)$ we mean $\Phi$ with the constant weight function $0$.

\begin{thm}\label{T:geomorthotopemu}
With $P$ and $\cA$ as in Theorem \ref{T:geomorthotope}, the number of integer points in $P \setm \bigcup\cA$ equals 
\begin{align*}
\chi_{(\Phi,0)}(m_1,\ldots,m_n) 
 = \sum_{B \in \Latb\Phi} \mu(\eset,B) \prod_{W_k \in \pi(B)} 
  \big( 1 + \min_{v_i\in W_k} [m_i+\phi(B_{v_it_k})] - g_k \big)^+  ,
\end{align*}
where $g_k$ is the maximum gain of a path in $B\ind W_k$.
\end{thm}

The right-hand side is $(-1)^nQ_\Phih(\bu,-1,0)$ with $u_{[a,\infty),(\infty,m_i]}$ set equal to $-|[a,m_i]|$, as in Theorem \ref{T:genlist}.
Note that $g_k$ equals the largest gain of paths that end at $t_k$, i.e., 
$$g_k = \max_{v_i\in W_k} \phi(B_{v_it_k}).$$

\begin{proof}
The lattice points to be counted are simply proper colorations in disguise.  
In Proposition \ref{L:genlist} take the list for $v_i$ to be $\H_i=[0,\infty)$ and the filter to be $M_i=(-\infty,m_i]$, and then sort through the definitions.  
E.g., from Equation \eqref{E:topvertexsw}, if $v_i$ is in the component $B_k$ of $B$, then $\eta_B(v_i) = \phi(B_{v_it_k})$.  
Also, letting $B_k:=B\ind W_k$, then $\H/B(W_k) = [ g_k, \infty )$ where $g_k$ is the largest gain of a path in $B_k$, and $M/B(W_k) = ( -\infty, \min_{v_i\in W_k} m_i+\phi(B_{v_it_k}) ]$.  
Hence, the factor in Proposition \ref{L:genlist} equals 
$$
\big|[ g_k, \min_{v_i\in W_k} m_i+\phi(B_{v_it_k}) ]\big| 
 = \big( 1 + \min_{v_i\in W_k} [m_i+\phi(B_{v_it_k})] - g_k \big)^+.
$$

Thus we have Theorem \ref{T:geomorthotope}.  
Theorem \ref{T:geomorthotopemu} follows by the formula for $\mu$ given at Proposition \ref{P:finlist}.
\end{proof}

The specialization to ordinary graphs is Corollary \ref{C:geomorthotope}.

For a second example, suppose that for each coordinate $x_i$ we have a finite list $L_i \subset \bbZ$ of possible values.  We want the number of lattice points in $L_1 \times \cdots \times L_n$ that are not in any hyperplane of $\cA$.  This generalizes the preceding theorem, but the viewpoint is different and the formula is more complex.

\begin{thm}\label{T:geomspec}
The number of these lattice points is given by the formula
\begin{align*}
\sum_{B \subseteq E: \text{ balanced}} (-1)^{|B|} \prod_{B_k} \, \big| \bigcap_{v_i\in V(B_k)} \big(L_i+\phi(B_{v_it_k})\big) \big|  ,
\end{align*}
where $L_i+g$ denotes the translate of $L_i$ by $g$, and the product is over all components $B_k$ of $B$.
\end{thm}

That is, we take the intersection of translates of the lists, governed by the gains of paths in the chosen balanced edge set $B$.  

\begin{proof}
The theorem follows directly from Proposition \ref{P:finlist} and the formula for $\eta_B(v_i)$.
\end{proof}

The ordinary-graph version of Theorem \ref{T:geomspec} extends Corollary \ref{C:geomorthotope} to general list coloring by giving an exact formula for the number of proper list colorations if the lists are finite.  (We are not aware of any such previously published formula.)  If $L(v)$ is a set associated to each vertex $v\in V$ and $W\subseteq V$, define $L(W):=\bigcap_{v\in W} L(v)$.

\begin{cor}[List coloring count]\label{C:geomspec}
Let $\Gamma$ be a graph in which each vertex $v$ is equipped with a finite list $L(v)$ of permitted colors.  The number of proper list colorations of $\Gamma$ is
\begin{align*}
&\sum_{B \in \Lat\Gamma} \mu_\Gamma(\eset,B)  \prod_{W \in \pi(B)}  | L(W) | 
= \sum_{B \subseteq E} (-1)^{|B|} \prod_{W \in \pi(B)} \, | L(W) |  .
\end{align*}
\end{cor}

\begin{proof}
We may assume without loss of generality that the list elements are integers.  Then this corollary is the specialization of Theorem \ref{T:geomspec} in the same way Corollary \ref{C:geomorthotope} specializes Theorem \ref{T:geomorthotopemu}.  We remark that this corollary can be proved directly by M\"obius inversion.
\end{proof}

We state one more theorem, a combination of the previous two.  Here we have for each coordinate a fixed list $L_i$ of nonnegative integral permitted values, which may be infinite, and a variable integral upper bound $m_i$.  Again we take the orthotope $P := [0,m_1]\times\cdots\times[0,m_n]$.

\begin{thm}\label{T:geomorthospec}
The number of points in $P \cap (L_1 \times \cdots \times L_n)$ but not in any of the hyperplanes of the arrangement $\cA$ equals 
\begin{align*}
\sum_{B \in \Latb\Phi} \mu(\eset,B) \prod_{B_k} 
  \big| \bigcap_{v_i\in V(B_k)} \big( (L_i \cap [0,m_i]) +\phi(B_{v_it_k}) \big) \big| ,
\end{align*}
where the product is over all components of $B$.
\end{thm}

\begin{proof}
Substitute $L_i \cap [0,m_i]$ for $L_i$ in Theorem \ref{T:geomspec}.
\end{proof}

When $L_i$ has finite complement in the nonnegative integers, then for sufficiently large variables $m_i$ this count is a polynomial in the variables, just as in Theorem \ref{T:geomorthotope}.

\subsection{Affinographic matrix subspaces}\label{affinomatrix}

Theorem \ref{T:orthotope} allows us to count integer matrices that are contained in an orthotope but not in any of a finite set of subspaces that are determined by affinographic equations.

Write $\bbZ^{n\times d}$ for the lattice of $n \times d$ integer matrices and $\bbR^{n\times d}$ for the real vector space that contains them; if $X$ is a matrix, we write $\bx_i = (x_{i1},\ldots,x_{id})$ for the $i$th row vector, an element of $\bbR^d$.  
An integral orthotope $[H,M]$, where $H$ and $M$ are integer matrices with $H \leq M$, is the convex polytope given by the constraints $H \leq X \leq M$ in $\bbR^{n\times d}$.  

We call a subspace determined by an equation of the form $\bx_j = \bx_i + \ba$ \emph{row-affinographic}, and \emph{integral} if $\ba$ is an integral vector in $\bbR^d$.  (The name ``affinographic'' comes from the fact that such a subspace is an affine translate of a graphic subspace, i.e., a subspace defined by lists of equal coordinates, in this case by the equation $\bx_j = \bx_i$.)  A finite set $\cS$ of such subspaces is an \emph{integral row-affinographic subspace arrangement}.

We want to know the number $N$ of integral matrices in an integral orthotope $[H,M]$ but not in any of the subspaces of $\cS$.  This number is given by Theorem \ref{T:orthotope}.  Rather than translate the theorem into purely geometrical language, which seems unnatural, we explain how to set up a weighted gain graph $\Phih$ to which it applies, thereby getting the formula
$$
N = \chi_\Phih(\bm_1,\ldots,\bm_n).
$$
There is one vertex for each row of the matrices; thus, $V = \{v_1,\ldots,v_n\}$.  There is one edge for each subspace; that with equation $\bx_j = \bx_i + \ba$ becomes an edge from $v_i$ to $v_j$ with gain $\ba$ (in that direction; the gain from $v_j$ to $v_i$ is $-\ba$).  The weight of $v_i$ is the principal dual ideal $\langle \bh_i \rangle^*$.  
An integral matrix $X$ in the orthotope becomes a coloration, the color of $v_i$ being the $i$th row vector $\bx_i$.  
It is now clear that an integral matrix that we wish to count is precisely the same as a proper coloration of $\Phih$ that satisfies the upper bound $(\bm_1,\ldots,\bm_n)$.

\section{Graphs with weights but not gains: Noble and Welsh generalized}\label{ordinary}

We think of a graph without gains as having all gains equal to the identity $\1$.  It is instructive to see what our results say here.  We assume there are no half or loose edges and we write $\G$ for $\Phi = (\G,1)$ to emphasize that, the gains being trivial, the only significant datum is the graph.  Since the graph is balanced, $b(S) = c(S)$ and $\pib(S)$ is a partition of $V$ for every edge set $S$.  

A \emph{$\fW$-weighted graph} is a pair $(\G,h)$ where $h: V \to \fW$.  There is no need for switching; thus contraction is ordinary graph contraction together with contraction of $h$ to 
$$
h(W) = \sum_{v_i \in W} h_i 
$$
for $W\in\pi(S)$, where summation means the semigroup operation and the subscript $S$ is superfluous because there is no switching.  
The total dichromatic polynomial becomes
\begin{align}
Q_{(\G,h)}(\bu,v,z) &= \sum_{S \subseteq E} v^{|S|-n+c(S)} \prod_{W \in \pi(S)} u_{h(W)}  \label{E:graphpolyq}
\intertext{with tree expansion}
&= \sum_T (v+1)^{\epsilon(T)} \sum_{\substack{F\subseteq T\\F\supseteq \II(T)}} \prod_{W\in\pi(F)}  u_{h(W)} .  
\label{E:graphpolyt}
\end{align}
Observe that $z$ drops out; thus we write $Q_{(\G,h)}(\bu,v)$ for this polynomial.  

These graphs with weights but no gains subsume the weighted graphs $(\G,\omega)$ of Noble and Welsh \cite{NW}, which have positive integral vertex weights.  Indeed, their work largely inspired our generalization to semigroup weights.  At first we had the total dichromatic polynomial only for weighted integral gain graphs with integral weights, but we compared their definitions to ours and noticed remarkable analogies.   Noble and Welsh's weights add: $\omega(W) = \sum_{w \in W} \omega(w)$, while the weights on weighted integral gain graphs maximize.  Our polynomial $Q_{(\G,h)}$ (for weighted integral gain graphs) and the polynomial $W_{(\G,\omega)}$ of \cite{NW} have virtually the same variables (if one makes simple substitutions) and satisfy the same Tutte relations (Ti--Tiii), initial conditions \eqref{E:initial}, and loop reduction identity \eqref{E:loop}.  
We had to suspect a common generalization.  This paper is the result.

The theorem without gains is stronger than our broader results because, unlike when there are gains, the dichromatic polynomial is universal.  

\begin{thm}\label{T:nwgen} 
Given an abelian semigroup $\fW$, the polynomial-valued function $(\G,h) \mapsto Q_{(\G,h)}(\bu,v)$ of $\fW$-weighted graphs is universal with the properties {\rm(Ti)}, {\rm(Tiii)}, {\rm(Tiv)}, \eqref{E:initial}, and \eqref{E:loop}; {\rm(Tii)} holds; and there is a tree expansion as in \eqref{E:graphpolyt}.
\end{thm}

\begin{proof}
The proof is like that of Noble and Welsh.
\end{proof}

(Added during revision in 2013:   Although the definitions are different, our weighted-graph dichromatic polynomial $Q_{(\G,h)}(\bu,v)$ is the same as the special case where all $\gamma_e=1$ of Ellis-Monaghan and Moffatt's $V$-polynomial \cite{EMM}, which they arrived at independently.  Equality is proved by universality, i.e., Theorem \ref{T:nwgen}, or by the subset expansion in \cite[Theorem 3.3]{EMM}, which coincides with our definition.)

The treatment of coloring in Section \ref{coloration} applies to ordinary graphs, without gains, simply by taking $\phi \equiv \1$, the identity.  The only differences are that every edge set is balanced, so $\Latb\Phi$ is the class of all closed edge sets (cf.\ the end of Section \ref{gaingraphs}), and that $\H/B(W)$ becomes simply $\H(W) = \bigcap_{v_i \in W} \H_i.$  Also, the weight semigroup $\fW$ need only be closed under intersection, as there is no group action constraining it.  

Taking weight semigroup $\bbZ^d$ as in Section \ref{intlatticelist} and treating the graph as having all zero gains, we have the following corollary of Theorem \ref{T:orthotope}; the notation is that of the theorem.

\begin{cor}\label{C:ordinaryorthotope}
Let $(\G,\H)$ be a weighted graph with weights $\H_i \in \fW_2$.  
Assume $\G$ has no loops or loose or half edges.  
For large enough $\bm$, $\chi_{(\G,\H)}(\bm)$ is a monic polynomial function of the $nd$ variables $m_{ik}$, having degree $1$ in each variable and highest-degree term $\prod_{i=1}^n \prod_{k=1}^d m_{ik}$.  Polynomiality holds when all $\bm_i \geq \bigvee_{j=1}^n \hat\bh_j .$
\end{cor}

When the weights are principal dual ideals $\langle \bh_i \rangle^*$, the lower bound on $\bm_i$ is $\bigvee_j \bh_j^-$.

(Added during revision.)  The special case of a graph colored from $\bbZ_{>0}$ with a finite list $\bar\H_i$ of excluded colors for each vertex $v_i$ was independently known to Carsten Thomassen; he mentioned that a direct proof by deletion and contraction is easy [personal communication to T.\ Zaslavsky, 21 January 2008].  That shows that the complications in our work come from gains, not coloring.  We conclude---generalizing Thomassen's observation---by stating the graph version of Theorem \ref{T:orthotope}.  Suppose each vertex $v_i$ has a weight $\bh_i = (h_{i1},\ldots,h_{id}) \in \bbZ^d$ and a finite exclusion set $\bar\H_i \subseteq \langle\bh_i\rangle^*$; a proper coloration $\bx:V\to\bbZ^d$ satisfies the requirement that the color $\bx_i$ of $v_i$ belongs to $\H_i = \langle\bh_i\rangle^* \setm \bar\H_i$.  $\Lat\G$ is the family of all closed edge sets.  

\begin{cor}\label{C:graph}
Let $\G$ be a simple graph with $\H$ and $\bar\H$ as described.  Let $\bm = (\bm_1,\ldots,\bm_n) \in (\bbZ^d)^n$.  
Then the number of proper colorations $\bx$ of $(\G,\H)$ such that $\bx \leq \bm$ is a piecewise polynomial function of $\bm$ with leading term $\prod_{i=1}^n \prod_{k=1}^d m_{ik}$, and when all upper bounds satisfy $\bm_j \geq \bigvee \bar\H_j$, it equals the piecewise polynomial 
$$
\bar p(\bm) := \sum_{B \in \Lat\G} \mu(\eset,B) \prod_{W \in \pi(B)}  \Big(  \prod_{k=1}^d  \big( \min_{v_i \in W} m_{ik} + 1 - \max_{v_i \in W} h_{ik} \big)  -  \big| \bigcup_{v_i \in W} \bar\H_i  \big|  \Big) .
$$
\end{cor}

Thomassen's case is where all $\bh_i=(1,\ldots,1)$ and all $\bm_i$ are equal and 1-dimensional.  Corollary \ref{C:graph} should have a relatively simple direct proof as well.


\end{document}